%% file: Steinhaus2V7_1.tex
\documentclass[a4paper,12pt]{amsart}

\usepackage{xcolor}
\definecolor{darkgreen}{RGB}{0,77,0}

\usepackage{amssymb,amsmath,amsfonts}
\usepackage{bm}

\usepackage{enumerate}
\usepackage{float}

\usepackage{pgf,tikz,pgfplots}
\pgfplotsset{compat=1.12}
\usepackage{mathrsfs}
\usetikzlibrary{arrows}

\usepackage{color}
\definecolor{red}{rgb}{1,0,0}

\definecolor{blu}{rgb}{0,0,0.8}

\definecolor{gre}{rgb}{0,0.7,0}

\newtheorem{thm}{Theorem}%
\newtheorem{prop}{Proposition}%
\theoremstyle{definition}

\theoremstyle{remark}
\theoremstyle{plain}

\newcommand{\N}{\mathbb N}
\newcommand{\R}{\mathbb R}
\newcommand{\Z}{\mathbb Z}
\newcommand{\Q}{\mathbb Q}
\newcommand{\rar}{\rightarrow}
\newcommand{\mc}{\mathcal}

\def\NN{{\mathbb N}}
\def\QQ{{\mathbb Q}}

\def\RR{{\mathbb R}}
\def\TT{{\mathbb T}}

\def\ZZ{{\mathbb Z}}
\def\SS{{\mathbb S}}
\def\Z{{\mathbb Z}}

\def\scrA{{\mathcal A}}

\def\scrC{{\mathcal C}}
\def\scrD{{\mathcal D}}

\def\scrF{{\mathcal F}}
\def\scrG{{\mathcal G}}

\def\scrH{{\mathcal H}}

\def\scrL{{\mathcal L}}
\def\scrM{{\mathcal M}}

\def\scrQ{{\mathcal Q}}

\def\scrS{{\mathcal S}}

\def\scrU{{\mathcal U}}

\def\e{\mathrm{e}}

\def\cl{{\operatorname{cl}}}

\def\SL{\operatorname{SL}}

\def\SO{\operatorname{SO}}

\def\GamG{\Gamma\backslash G}

\numberwithin{equation}{section}

\begin{document}

\title{A five distance theorem for Kronecker sequences}
\author{Alan Haynes, Jens Marklof}

\thanks{AH: Research supported by NSF grant DMS 2001248.\\
\phantom{A..}JM: Research supported by EPSRC grant EP/S024948/1.\\
\phantom{A..}MSC 2020: 11J71, 37A44.}
\date{14 September 2020/30 June 2021}


\begin{abstract}
The three distance theorem (also known as the three gap theorem or Steinhaus problem) states that, for any given real number $\alpha$ and integer $N$, there are at most three values for the distances between consecutive elements of the Kronecker sequence $\alpha, 2\alpha,\ldots, N\alpha$ mod 1. In this paper we consider a natural generalisation of the three distance theorem to the higher dimensional Kronecker sequence $\vec\alpha, 2\vec\alpha,\ldots, N\vec\alpha$ modulo an integer lattice. We prove that in two dimensions there are at most five values that can arise as a distance between nearest neighbors, for all choices of $\vec\alpha$ and $N$. Furthermore, for almost every $\vec\alpha$, five distinct distances indeed appear for infinitely many $N$ and hence five is the best possible general upper bound. In higher dimensions we have similar explicit, but less precise, upper bounds. For instance in three dimensions our bound is 13, though we conjecture the truth to be 9. We furthermore study the number of  possible distances from a point to its nearest neighbor in a restricted cone of directions. This may be viewed as a generalisation of the gap length in one dimension. For large cone angles we use geometric arguments to produce explicit bounds directly analogous to the three distance theorem. For small cone angles we use ergodic theory of homogeneous flows in the space of unimodular lattices to show that the number of distinct lengths is (a) unbounded for almost all $\vec\alpha$ and (b) bounded for $\vec\alpha$ that satisfy certain Diophantine conditions.
\end{abstract}

\maketitle




\section{Introduction}\label{sec:Results}

Consider a finite set $S_N$ comprising $N$ distinct points $\xi_1,\ldots,\xi_N$ on the unit torus $\TT=\R/\Z$. The points in $S_N$ partition $\TT$ into $N$ intervals, representing the {\em gaps} of $S_N$. We denote by $\delta_{n,N}$ the size of the $n$th gap, i.e., the distance between $\xi_n$ and its nearest neighbor {\em to the right}. We denote by $g_N=|\{\delta_{n,N} \mid 1\leq n\leq N \}|$ the number of distinct gap sizes. For a generic choice of $S_N$ one has $g_N=N$, since all gap lengths are generically distinct. A striking observation, known as the three distance (or three gap) theorem, is that for the Kronecker sequence $\xi_n=n\alpha+\Z$, one has $g_N\leq 3$, for any $\alpha\in\RR$ and $N\in\NN$; see \cite{Sos1957,Sos1958,Sura1958,Swie1959} for the original proofs and \cite{FrieSos1992,HaynKoivSaduWalt2016,Lang1991,MarkStro2017,Rave1988,Slat1950,Slat1967} for alternative approaches. Natural extensions to return maps for billiards in rectangles and interval exchange transformations are discussed in \cite{Don2012,Flor2008,Flor2009,FlorFlor2009} and \cite{Taha2017}, respectively.
There are various generalisations of the three gap theorem to higher dimensions, several of which require Diophantine conditions on the choice of parameter. In the present paper we discuss natural extensions of the three distance theorem to higher dimensional Kronecker sequences, which represent translations of a multidimensional torus by a vector $\vec\alpha$. We will here consider nearest neighbor distances as well as distances to neighbors in restricted directions. The former may be viewed as a special case of the setting of Biringer and Schmidt \cite{BiriSchm2008}, who considered the number of distinct nearest neighbor distances for an orbit generated by an isometry of a general Riemannian manifold. If distances are measured by the maximum norm rather than a Riemannian metric, Chevallier \cite[Corollaire 1.2]{Chev1996} showed that there are at most five distinct distances for Kronecker sequences on two-dimensional tori. Related studies in this context include the papers by Chevallier \cite{Chev1997,Chev2000,Chev2014} and by Vijay \cite{Vija2008}. For other higher dimensional variants of the three distance theorem, see \cite{BertKim2018,BlehHommJiRoedShen2012,Bosh1991,Bosh1992,ChuGra1976,CobeGrozVajaZaha2002,Dyson,FraeHolz1995,GeelSimp1993,HaynKoivSaduWalt2016,HaynMar2020, HaynRoed2020,Liang1979}.

Our setting is as follows. Let $\scrL$ be a unimodular lattice in $\RR^d$ (one example to keep in mind is $\Z^d$) and consider a point set $S_N=\{\xi_1,\ldots,\xi_N\}$ on the $d$-dimensional torus $\TT^d=\R^d/\scrL$. The point set $S_N$ that we are interested in is the $d$-dimensional Kronecker sequence,
\begin{equation}
S_N=S_N(\vec\alpha,\scrL)=\{\xi_n=n\vec\alpha +\scrL \mid 1\le n\le N\}\subset\TT^d,
\end{equation}
for given $\vec\alpha\in\RR^d$. Note that the $\xi_n$ are not necessarily distinct if $\vec\alpha\in\QQ\scrL$; in this case we remove all multiple occurrences, and as a result the number of elements in $S_N$ remains bounded as $N\to\infty$.

Let $\delta_{n,N}$ be the distance of $\xi_n$ to its nearest neighbor with respect to the standard flat Riemannian metric on $\TT^d$. As above, $g_N=g_N(\vec\alpha,\scrL)$ denotes the number of distinct nearest neighbor distances $\delta_{n,N}$. Biringer and Schmidt \cite{BiriSchm2008} proved that, for all choices of $\scrL, \vec{\alpha},$ and $N$,
\begin{equation}
	g_N(\vec\alpha,\scrL)\le 3^d+1.
\end{equation}
Our first theorem improves this bound, and also gives the best possible result in dimension $d=2$.
\begin{thm}\label{thm.HighDBoundedGaps}
For every unimodular lattice $\scrL$, $\vec{\alpha}\in\R^d$ and $N\in\N$ we have that
\begin{equation}
g_N(\vec\alpha,\scrL) \le 
\begin{cases}
5&\text{if}~d=2,\\ 
\sigma_d+1&\text{if}~d\ge 3,
\end{cases}
\end{equation}
where $\sigma_d$ is the kissing number for $\R^d$.
\end{thm}
Recall that the \textit{kissing number} $\sigma_d$ for $\R^d$ is the maximum number of non-overlapping spheres of radius one in $\R^d$ which can be arranged so that they all touch the unit sphere in exactly one point. The study of kissing numbers has a long and interesting history, and is connected to many areas of mathematics (see \cite{BoyvDoduMusi2012} and \cite{Musi2008} for surveys of results). It is interesting to note that the three distance theorem in dimension $d=1$ is compatible with the bound $g_N\leq \sigma_1+1=3$, but that already in dimension $d=2$ this becomes suboptimal, since here $\sigma_2+1=7>5$. A table of known bounds for kissing numbers in dimensions $d\le 24$ is provided in Figure \ref{fig.KissNums}. Our corresponding bounds for $g_N$ in dimensions $d=3,\ldots,10$ are therefore 13, 25, 46, 79, 135, 241, 365, 555. We do not claim that these bounds are optimal for any dimension $d\geq 3$. In particular, we conjecture that $g_N(\vec\alpha,\scrL) \le 9$ for $d=3$. This conjecture is based on numerical experiments by Dettmann \cite{Dett2021} that produced no more than 9 distinct gaps. Figures \ref{fig.5Gaps1} and \ref{fig.3Dgaps1} show examples where 5 (for $d=2$) and 7 (for $d=3$) distinct distances are obtained.

For general $d$, it follows from Theorem \ref{thm.HighDBoundedGaps} together with an estimate for $\sigma_d$ due to Kabatiansky and Levenshtein (see \cite[Theorem 4, Corollary 1]{KabaLeve1978} or \cite[Chapter 9]{ConwSloa1999}) that
\begin{equation}
g_N(\vec\alpha,\scrL)\le 2^{0.401d(1+o(d))}\quad\text{as}\quad d\rar\infty.
\end{equation}
The rate of convergence of the $o(d)$ term in this estimate can be made more precise and explicit (non-asymptotic) upper bounds can also be obtained by applying \cite[Equation (52)]{KabaLeve1978}.

\begin{figure}\label{fig.KissNums}
\centering
\begin{tabular}{|c|c|}
	\hline
	\hspace*{6bp}$d$\hspace*{6bp}& Upper bound   \\
	& for $\sigma_d$\\
	\hline
	1 & \textbf{2} \\
	\hline
	2 & \textbf{6} \\
	\hline
	3 & \textbf{12} \\
	\hline
	4 & \textbf{24} \\
	\hline
	5 & 45 \\
	\hline
	6 & 78 \\
	\hline
	7 & 134 \\
	\hline
	8 & \textbf{240} \\
	\hline
	9 & 364 \\
	\hline
	10 & 554 \\
	\hline
	11 & 870 \\
	\hline
	12 & 1357 \\
\hline
\end{tabular}
\begin{tabular}{|c|c|}
	\hline
	\hspace*{6bp}$d$\hspace*{6bp}& Upper bound   \\
	& for $\sigma_d$\\	
	\hline
	13 & 2069 \\
	\hline
	14 & 3183 \\
	\hline
	15 & 4866 \\
	\hline
	16 & 7355 \\
	\hline
	17 & 11072 \\
	\hline
	18 & 16572 \\
	\hline
	19 & 24812 \\
	\hline
	20 & 36764 \\
	\hline
	21 & 54584 \\
	\hline
	22 & 82340 \\
	\hline
	23 & 124416 \\
	\hline
	24 & \textbf{196560} \\
	\hline
\end{tabular}
	\caption{A table of known bounds for kissing numbers. All bounds are taken from \cite{BoyvDoduMusi2012}, and bounds listed in bold face are known to be best possible.}
\end{figure}


Values of $\vec\alpha$ for which $g_N(\vec\alpha,\scrL)= 5$ are surprisingly rare. Our first computer search, which took 1000 randomly and uniformly selected numbers $\vec\alpha\in [0,1)^2$ and checked $g_N(\vec\alpha,\ZZ^2)$ for all $N\le 10^4$, found only five values of $\vec\alpha$ for which there was an $N$ with $g_N(\vec\alpha,\ZZ^2)=5$. 

Nevertheless, as we will now see, every numerical example gives rise to a lower bound for an infinite sequence of $N$ and for almost every $\vec\alpha$. In what follows, we say a sequence $N_1< N_2 < N_3 <\ldots$ of integers is {\em sub-exponential} if
\begin{equation}\lim_{i\to\infty} \frac{N_{i+1}}{N_i}=1 .\end{equation}

\begin{thm}\label{prop:BestPoss}
Let $\scrL$ and $\scrL_0$ be unimodular lattices. There is a $P\subset\RR^d$ of full Lebesgue measure, such that for every $\vec\alpha\in P$, $\vec\alpha_0\in\RR^d$, and for every sub-exponential sequence $(N_i)_i$, we have
	\begin{equation}
	\limsup_{i\rar\infty} g_{N_i}(\vec\alpha,\scrL) \geq \sup_{N\in\NN} g_{N}(\vec\alpha_0,\scrL_0) .
	\end{equation}
\end{thm}

\begin{figure}
	\centering
	\def\svgwidth{0.49\columnwidth}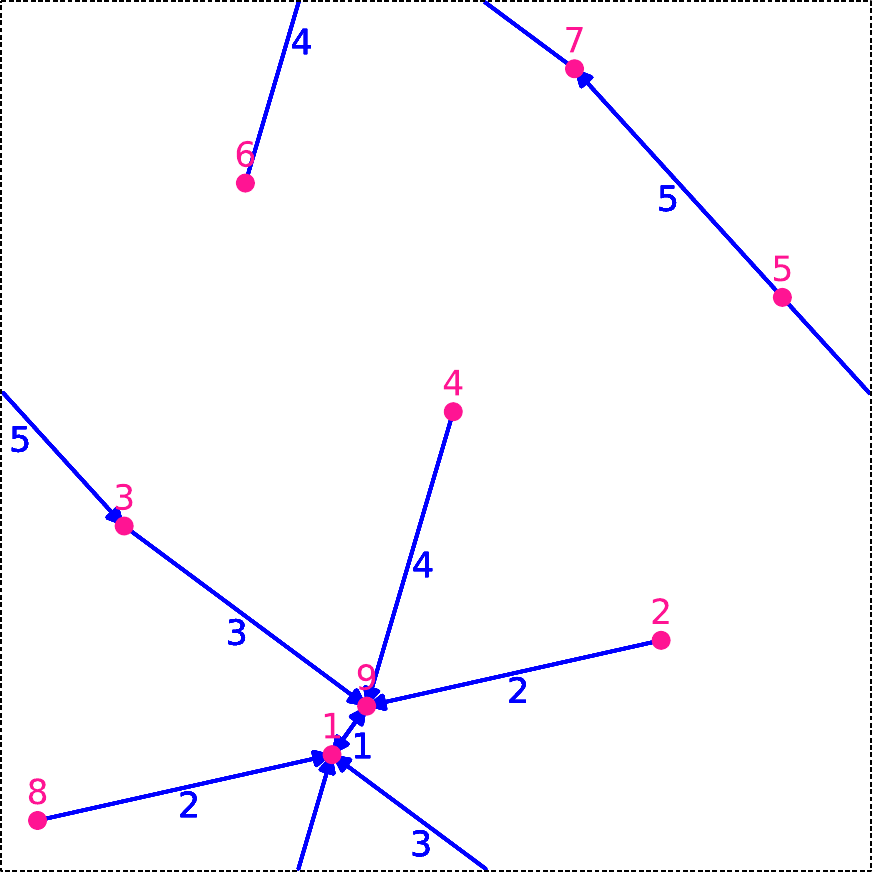
	\def\svgwidth{0.49\columnwidth}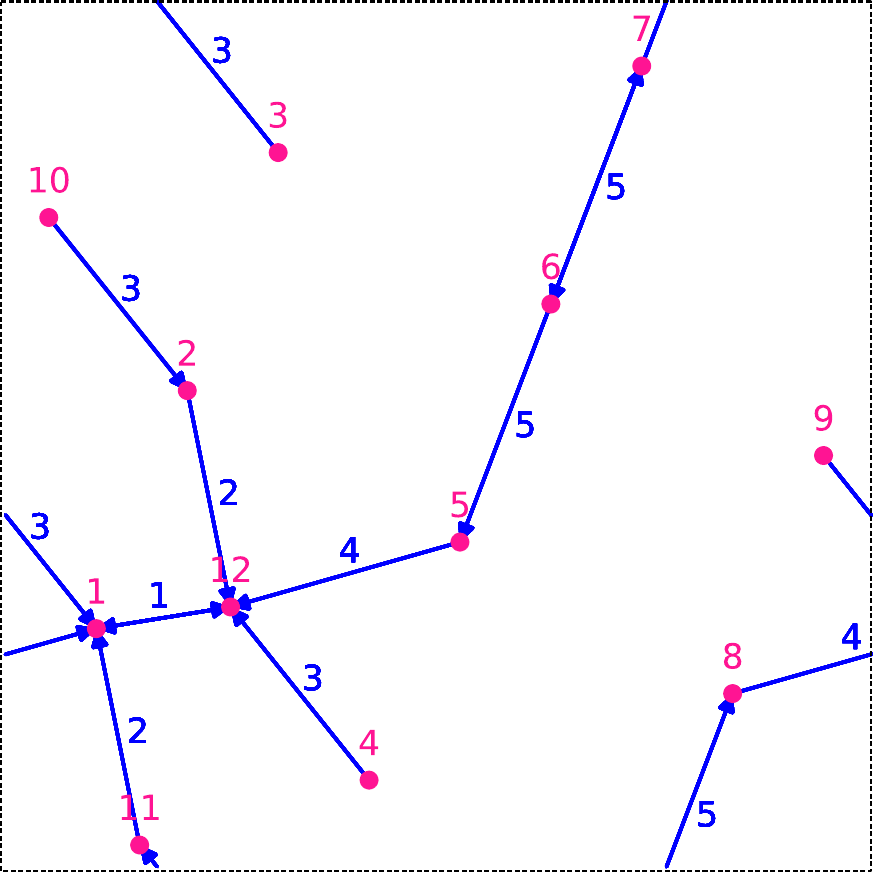\\
	{\footnotesize $\vec\alpha=(0.38,0.132)$, $N=9$ \hspace*{110bp} $\vec\alpha=(0.105,0.275)$, $N=12$}
	\caption{The nearest neighbor graph for the Kronecker sequence $n\vec\alpha+\ZZ^2$ ($n=1,\ldots,N$) in the torus $\RR^2/\ZZ^2$ (the unit square with opposite sides identified), for $N=9$ (left) and $N=12$ (right) and different choices of $\vec\alpha$. The vertex representing $n\vec\alpha+\ZZ^2$ is labeled by $n$ and colored in pink. The blue directed edges point from a vertex to its nearest neighbour(s). The blue edge labels correspond to the indices of each of the five distinct distances. Note that in each example there is a vertex with two nearest neighbours: vertex $n=5$ on the left and $n=6$ on the right.}\label{fig.5Gaps1}
\end{figure}

In dimension $d=2$, the choice 
	$\vec\alpha_0=\left(\tfrac{19}{50},\tfrac{33}{250}\right)$, $N=9$ and $\scrL_0=\ZZ^2$
	(this is the example in Figure \ref{fig.5Gaps1}, left) produces precisely five distinct distances given by
	\begin{align*}
		&\delta_{1,N}=\tfrac{\sqrt{74}}{125}  \approx 0.0688, \quad 
		\delta_{2,N}=\tfrac{\sqrt{\frac{3793}{2}}}{125} \approx  0.3484 ,\quad 
		\delta_{3,N}=\tfrac{\sqrt{1901}}{125}  \approx 0.3488, \\
		&\hspace*{55bp}\delta_{4,N}=\tfrac{\sqrt{\frac{157}{2}}}{25}  \approx  0.3544,\quad 
		\delta_{5,N}=\tfrac{3 \sqrt{221}}{125} \approx  0.3568.
	\end{align*}
Moreover, for $d= 3$, $\vec\alpha_0=(\tfrac{46}{125}, \tfrac{107}{500}, \tfrac{43}{500})$, $N=15$ and $\scrL_0=\ZZ^3$ (cf.~Figure \ref{fig.3Dgaps1}), we have seven distinct distances,
\begin{align*}
&\delta_{1,N}=\tfrac{17\sqrt{\frac{7}{2}}}{125}\approx 0.2544,\quad
\delta_{2,N}=\tfrac{\sqrt{\frac{13513}{2}}}{250}\approx 0.3288,\quad 
\delta_{4,N}=\tfrac{21\sqrt{\frac{37}{2}}}{250}\approx 0.3613, \\
&\delta_{5,N}=\tfrac{\sqrt{\frac{177}{2}}}{25}\approx 0.3763,\quad
\delta_{6,N}=\tfrac{\sqrt{\frac{19237}{2}}}{250}\approx 0.3923,\quad 
\delta_{7,N}=\tfrac{\sqrt{2866}}{125}\approx 0.4283,\\
&\hspace*{127bp}\delta_{8,N}=\tfrac{\sqrt{\frac{23577}{2}}}{250}\approx 0.4343.
\end{align*}
Applying these data to Theorem \ref{prop:BestPoss}, combined with the upper bound of 5 in Theorem \ref{thm.HighDBoundedGaps}, we immediately obtain the following result.

\begin{thm}\label{thm.2DNearNeigBestPoss}
Let $d=2$ or $3$. For any unimodular lattice $\scrL$ in $\R^d$, there is a set $P\subset\RR^d$ of full Lebesgue measure, such that for every $\vec\alpha\in P$, and for every sub-exponential sequence $(N_i)_i$, we have that
	\begin{equation}
	\limsup_{i\rar\infty} g_{N_i}(\vec\alpha,\scrL)
	\begin{cases}
	= 5 & \text{if $d=2$,}\\
	\geq 7 & \text{if $d= 3$}.
	\end{cases}
	\end{equation}
\end{thm}

Thus the upper bound of 5 in Theorem \ref{thm.HighDBoundedGaps} is indeed optimal for every $\scrL$, almost every $\vec\alpha$ and infinitely many $N$.

\begin{figure}
	\centering
	\def\svgwidth{0.6\columnwidth}
	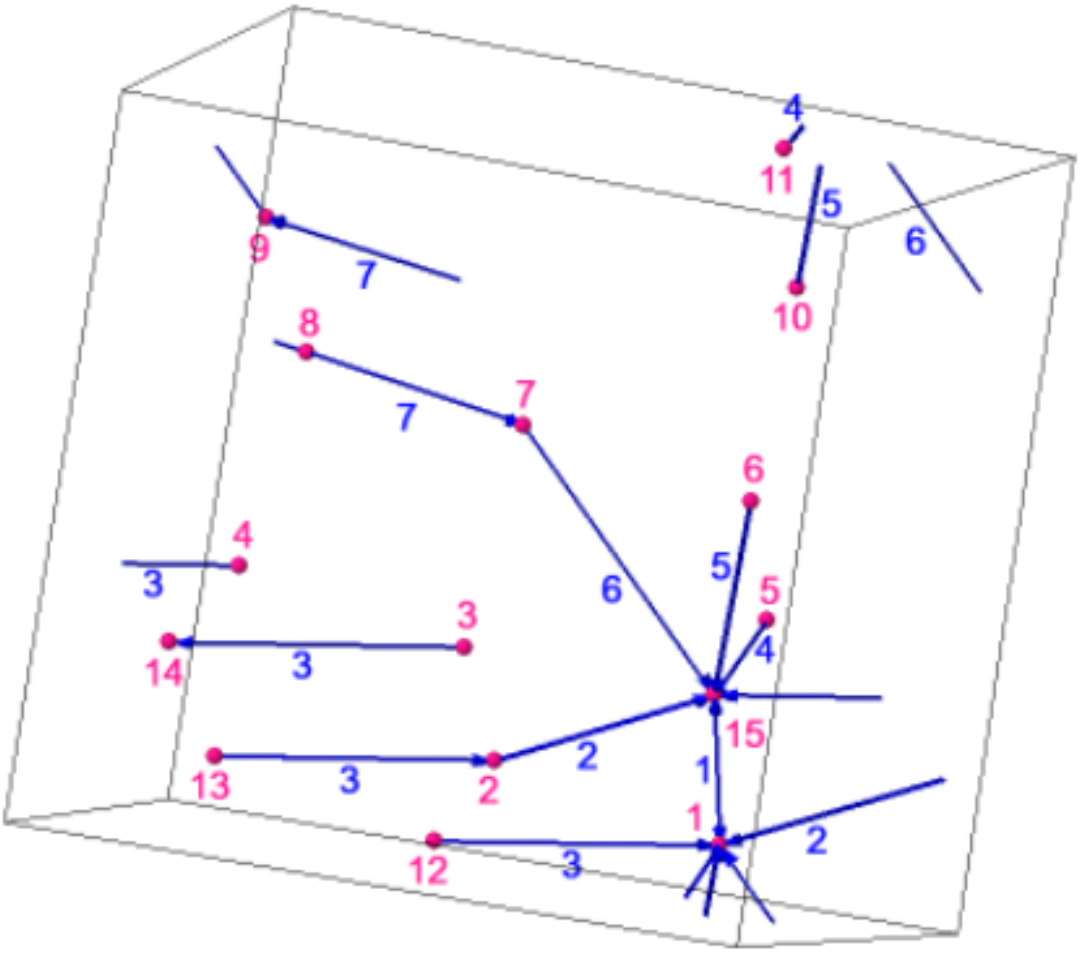
	\caption{The nearest neighbor graph for the Kronecker sequence $n\vec\alpha+\ZZ^3$ ($n=1,\ldots,15$) in the torus $\R^3/\Z^3$ (the unit cube with opposite faces identified), for $\vec\alpha=(\frac{46}{125}, \frac{107}{500}, \frac{43}{500})$. The vertex representing $n\vec\alpha+\ZZ^3$ is labeled by $n$ and colored in pink. The blue directed edges point from a vertex to its nearest neighbour(s). The blue edge labels correspond to the indices of each of the seven distinct distances. \label{fig.3Dgaps1}}
\end{figure}

In dimension $d=1$ nearest neighbour distances do not necessarily coincide with the set of gap lengths, as gaps are the distance to the nearest neighbor in a fixed direction. To generalise this interpretation of a gap to higher dimensions, fix a subset $\scrD$ of the unit sphere $\SS_1^{d-1}$, and denote by $\delta_{n,N}(\scrD)$ the distance from $\xi_n$ to its nearest neighbor {\em in the direction of $\scrD$}. More precisely, denote by $\vec\xi_n\in\RR^d$ a fixed representative of the coset $\xi_n\bmod \scrL$ so that $\xi_n=\vec\xi_n+\scrL$ and define
\begin{equation}\label{NN11}
\delta_{n,N}(\scrD) = \min\{| \vec\xi_m-\vec\xi_n + \vec\ell| \mid \vec\xi_m-\vec\xi_n +\vec\ell \in \RR_{>0}\scrD,\; 1\leq m\leq N,\; \vec\ell\in\scrL \} ,
\end{equation}
where $|\,\cdot\,|$ is the standard Euclidean norm in $\RR^d$. For $\scrD=\SS_1^{d-1}$ we recover the nearest neighbor distance
\begin{equation}
\delta_{n,N}(\SS_1^{d-1}) = \delta_{n,N}=\min\{| \vec\xi_m-\vec\xi_n + \vec\ell|>0 \mid 1\leq m\leq N,\; \vec\ell\in\scrL \}. 
\end{equation}
In particular, note that the nearest neighbor of $\xi_n$ might be $\xi_n$ itself; in this case $\delta_{n,N}=|\vec\ell|$ for suitable non-zero $\ell\in\scrL$. 

As an illustration consider dimension $d=1$. The circle of radius one is $\SS_1^0=\{-1,1\}$ and the choice $\scrD=\{1\}$ produces  the distance to the nearest neighbor to the right (which leads to the classical three distance theorem). On the other hand $\scrD=\{-1,1\}$ yields the distance to the nearest neighbor (in both directions). Since we are in dimension one, the set of distances to nearest neighbors is contained in the set of distances to nearest neighbors to the right, but it is not necessarily equal to it. However, for the case of the Kronecker sequence, the three distance theorem also holds also for nearest neighbors. In other words, there are examples of $\alpha$ and $N$ for which the number of distinct nearest neighbor distances is equal to three (e.g. take $\alpha=e$ and $N=5$).

The central object of our study is the number $g_N(\scrD)$ of distinct nearest neighbor distances in direction $\scrD\subseteq\SS_1^{d-1}$, 
\begin{equation}
g_N(\scrD)  = |\{ \delta_{n,N}(\scrD) \mid 1\le n \le N \}|,
\end{equation}
for $d\ge 2$. We will also write $g_N(\scrD)  = g_N(\scrD,\vec\alpha,\scrL)$ to highlight the dependence on vector and lattice.

We first present detailed results for dimension $d=2$. In this case $\SS_1^{d-1}$ is the unit circle in $\RR^2$.
The following theorem deals with the case when $\scrD\subseteq\SS_1^1$ is an interval of arclength $\tau>\pi$. (The case $\tau=2\pi$ has  already been covered in Theorem \ref{thm.HighDBoundedGaps} and we include it here for completeness.)

\begin{thm}\label{thm.2DBoundedGaps}
Let $d=2$, and assume $\scrD\subseteq\SS_1^1$ is a half-open interval of arclength $\tau>\pi$.
Then for any unimodular lattice $\scrL$, $\vec{\alpha}\in\R^2$ and $N\in\N$ we have that
\begin{equation}\label{eqn.gNUpperBd}
g_N(\scrD,\vec\alpha,\scrL)\le\begin{cases}
5&\text{if}\quad \tau=2\pi ,\\
9&\text{if}\quad5\pi/3<\tau<2\pi,\\
8&\text{if}\quad\tau=5\pi/3,\\
9&\text{if}\quad4\pi/3<\tau<5\pi/3,\\
8&\text{if}\quad\tau=4\pi/3,\\
12+2\left\lfloor\frac{\sin (\tau/2+\pi/6)}{\sin(\tau-\pi)}\right\rfloor&\text{if}\quad\pi<\tau<4\pi/3.
\end{cases}
\end{equation}
\end{thm}

Related results in various settings have been obtained independently by Chevallier \cite{Chev1997,Chev2000,Chev2014} and Vijay \cite{Vija2008}, but the precise description of the nearest neighbor problem which we give here does not appear to have been considered. However, the problem studied by Vijay in his paper ``Eleven Euclidean distances are enough'' \cite{Vija2008} is roughly comparable to the special case of $\tau=3\pi/2$ in Theorem \ref{thm.2DBoundedGaps} above.

By way of contrast with the upper bounds above, our next result demonstrates that the restriction in Theorems \ref{thm.HighDBoundedGaps} and \ref{thm.2DBoundedGaps} to a ``large'' set of directions $\scrD$ is essential. Note that the upper bound for $g_N$ in the final case of Theorem \ref{thm.2DBoundedGaps} tends to infinity as $\tau\rar\pi^+$. In fact, in dimension $d=2$ intervals $\scrD\subset\SS_1^1$ of lengths $\tau<\pi$ produce unbounded numbers of distinct distances, for almost every $\vec\alpha$. This is part of the content of the following theorem, which also deals with analogous regions $\scrD$ in higher dimensions.

\begin{thm}\label{thm:one}
Let $d\geq 2$ and $\scrL$ a unimodular lattice.
There exists a set $P\subset\RR^d$ of full Lebesgue measure, such that for every $\scrD\subset\SS_1^{d-1}$ with non-empty interior and closure contained in an open hemisphere, for every $\vec\alpha\in P$, and for every sub-exponential sequence $(N_i)_i$, we have
\begin{equation} \label{diverge-bdd}
\sup_i g_{N_i}(\scrD,\vec\alpha,\scrL)=\infty ,\qquad \liminf_i g_{N_i}(\scrD,\vec\alpha,\scrL)<\infty .
\end{equation}
\end{thm}

Our final observation is that the finite distance phenomenon is recovered in any dimension, for general test sets $\scrD$, if we impose Diophantine conditions on $\vec\alpha$. We say that $\vec\alpha\in\RR^d$ is {\em badly approximable by $\QQ\scrL$} if there is a constant $c>0$ such that $|n \vec\alpha - \vec\ell|_\infty > c n^{-1/d}$ for all $\vec\ell\in\scrL$, $n\in\NN$. Here $|\;\cdot\;|_\infty$ denotes the maximum norm.

\begin{thm}\label{thm:two}
Let $d\geq 2$, $\scrL$ a unimodular lattice and $\vec\alpha\in\RR^d$ badly approximable by $\QQ\scrL$. 
For $\scrD\subseteq\SS_1^{d-1}$  with non-empty interior, we have that
\begin{equation}\sup_{N\in\NN} g_{N}(\scrD,\vec\alpha,\scrL)<\infty.\end{equation}
\end{thm}

For comparison, the more precise bounds in Theorems \ref{thm.HighDBoundedGaps} and \ref{thm.2DBoundedGaps} hold for all $\vec\alpha$, but only for a restricted class of $\scrD$.

To relate the above Diophantine condition on $\vec\alpha$ to the standard notion of {\em badly approximable by $\QQ^d$} (which corresponds to the special case $\scrL=\ZZ^d$), take $M_0\in\SL(d,\RR)$ so that $\scrL=\ZZ^d M_0$. We then see that $\vec\alpha\in\RR^d$ is badly approximable by $\QQ\scrL$ if and only if $\vec\alpha_0=\vec\alpha M_0^{-1}$ is badly approximable by $\QQ^d$. (The positive constants $c$ appearing in both definitions are not necessarily the same.) 
Furthermore, by Khintchine's transference principle (see the Corollary to Theorem II in \cite[Chapter V]{Cass1957}), the vector $\vec\alpha_0$ is badly approximable by $\QQ^d$ if and only if there is $c>0$ such that $\| \vec m\cdot\vec\alpha_0 \|_{\RR/\ZZ} > c | \vec  m |^{-d}$ for all non-zero $\vec m\in\ZZ^d$. Here $\| x \|_{\RR/\ZZ}=\min_{k\in\ZZ} | x+k|$ denotes the distance to the nearest integer.

The key strategy of the proofs of the above theorems is to express the quantities $S_N(\vec\alpha)$, $\delta_{n,N}(\scrD)$ and $g_N(\scrD)$ in terms of functions on the space $\SL(d+1,\ZZ)\setminus\SL(d+1,\RR)$ of unimodular lattices in $\RR^{d+1}$. This is explained in detail in Section \ref{sec:Lattices}. Once this connection is established, the proofs of Theorems  \ref{thm.HighDBoundedGaps} and \ref{thm.2DBoundedGaps} reduce to geometric arguments involving lattices and sphere coverings, which are laid out in Sections \ref{sec.2Dprep}-\ref{sec.2DBddPfPart2} in dimension $d=2$ and in Section \ref{sec.HighDBddPf} for dimensions $d\ge 3$. The proofs of Theorems \ref{prop:BestPoss}, \ref{thm:one}, and \ref{thm:two} require upper and lower bounds for the relevant functions on the space of lattices, combined with the same ergodic-theoretic arguments used in \cite{HaynMar2020}. This material is presented in Section \ref{uplow}.\vspace*{10bp}

\noindent Acknowledgments: We would like to thank Nicolas Chevallier for helpful comments, and Felipe Ramirez and Carl Dettmann for discussions that led to an improvement of our bounds in Theorem \ref{thm.HighDBoundedGaps} in dimension $d\geq 3$. We would also like to thank the anonymous referees who carefully read our paper and provided many useful comments. The images in Figures \ref{fig.5Gaps1}-\ref{fig.GapsProp1} were generated using the computer software packages SageMath, Jmol, GeoGebra, and Inkscape. Finally, we would like to thank Timothy Haynes for his help in optimizing our Python code, which aided in the discovery of the examples illustrated in Figures \ref{fig.5Gaps1} and \ref{fig.3Dgaps1}.

\section{Reformulation in terms of lattices}\label{sec:Lattices}

This section follows the approach developed for the three gap theorem \cite{MarkStro2017} and higher dimensional variants  concerning gaps in values taken by linear forms, and hitting times for toral rotations \cite{HaynMar2020}.

By substituting $k=m-n$ in equation \eqref{NN11}, we find that the distance from $\xi_n$ to its nearest neighbor in the direction of $\scrD$ is given by
\begin{equation}\label{eqn.DefDelta_n,N}
\begin{split}
\delta_{n,N}(\scrD) & = \min\{| k\vec\alpha + \vec\ell | \mid k\vec\alpha + \vec\ell \in \RR_{>0}\scrD,\; -n< k\leq N-n,\; \vec\ell\in\scrL \} \\
& = \min\{| k\vec\alpha + \vec\ell | \mid k\vec\alpha + \vec\ell \in \RR_{>0}\scrD,\; -n< k< N_+-n,\; \vec\ell\in\scrL \} ,
\end{split}
\end{equation}
where $N_+:=N+\tfrac12$.
Select $M_0\in\SL(d,\RR)$ so that $\scrL=\ZZ^d M_0$, and let
\begin{equation}\label{eqn.A_NDef}
A_N(\vec\alpha)=A_N(\vec\alpha,\scrL)=\begin{pmatrix} 1 & 0 \\ 0 & M_0  \end{pmatrix} \begin{pmatrix} 1 & \vec{\alpha} \\ 0 & \bm{1}_d  \end{pmatrix} \begin{pmatrix} N^{-1} & 0 \\ 0 & N^{1/d}\bm{1}_d\end{pmatrix}.
\end{equation}
Then we have that
\begin{multline}
\delta_{n,N}(\scrD) = N_+^{-1/d} \min\bigg\{| \vec v | \;\bigg|\;  (u,\vec v)\in\ZZ^{d+1} A_{N_+}(\vec\alpha),\\
 -\frac{n}{N_+}< u < 1-\frac{n}{N_+},
\;\vec v \in \RR_{>0}\scrD \bigg\} ,
\end{multline}
for all $1\leq n\leq N$.
To cast this in a more general setting, let $G=\SL(d+1,\RR)$ and $\Gamma=\SL(d+1,\ZZ)$. Then for general $M\in G$ and $t\in (0,1)$, define
\begin{equation} \label{QMt}
\scrQ_\scrD(M,t) = \big\{ (u,\vec v)\in\ZZ^{d+1} M\;\big| \; -t< u < 1-t, \;\vec v \in \RR_{>0}\scrD \big\}
\end{equation}
and
\begin{equation}\label{FMt}
F_\scrD(M,t) = \min\big\{| \vec v | \;\big|\;  (u,\vec v)\in\scrQ_\scrD(M,t) \big\} .
\end{equation}
In this notation it is clear that
\begin{equation}\label{NN1102}
\delta_{n,N}(\scrD) = N_+^{-1/d} F_\scrD\bigg(A_{N_+}(\vec\alpha),\frac{n}{N_+} \bigg) .
\end{equation}
Before proceeding further, we first establish the following basic result.

\begin{prop}\label{WellProp}
If $\scrD\subseteq\SS_1^{d-1}$ has non-empty interior, then $F_\scrD$ is well-defined as a function $\GamG\times (0,1) \to \RR_{>0}$.
\end{prop}

\begin{proof}
We first show that the set $\scrQ_\scrD(M,t)$
is non-empty for all $M\in G$ and $t\in (0,1)$. Fix any $\vec w\in\scrD^\circ$, and denote by $\Sigma_\perp\subset\RR^d$ the $(d-1)$-dimensional hyperplane perpendicular to $\vec w$. Denote by $\vec v_\perp$ the orthogonal projection of $\vec v$ to $\Sigma_\perp$.
Given $(M,t)$, let $\epsilon>0$ be sufficiently small so that (i) $\epsilon<\min\{t,1-t\}$, (ii) there is no non-zero lattice point in $\ZZ^{d+1} M$ within $\epsilon$-distance to the origin. Furthermore, fix $\delta>0$ sufficiently small so that
\begin{equation}\label{Hset}
\{ \vec v \in\RR^d \mid \vec w\cdot\vec v>\epsilon, \;|\vec v_\perp|<\delta \} \subset \RR_{>0}\scrD.
\end{equation}
Such a $\delta$ exists since $\vec w\in\scrD^\circ$.
By construction
\begin{equation}\label{HSET}
\{ (u,\vec v) \in\ZZ^{d+1} M \mid |u|<\epsilon,\;\vec w\cdot\vec v>\epsilon, \; |\vec v_\perp|<\delta \} \subset \scrQ_\scrD(M,t)
\end{equation}
In view of Minkowski's theorem, the symmetric, convex set $\{ (u,\vec v) \in\RR^{d+1} \mid |u|<\epsilon,\;  |\vec v_\perp|<\delta \}$ contains a non-zero element of $\ZZ^{d+1} M$. By construction there is no non-zero lattice point in $\ZZ^{d+1} M$ within $\epsilon$-distance to the origin, and furthermore $\ZZ^{d+1} M$ is symmetric under reflection at the origin. We can therefore conclude that the set in
\eqref{HSET}, and therefore also $\scrQ_\scrD(M,t)$, is non-empty. Hence by the uniform discreteness of $\ZZ^{d+1} M$ the minimum value in the definition of $F_\scrD$ exists. 

Finally, note that $F_\scrD(\gamma M,t)= F_\scrD(M,t)$ for $\gamma\in\Gamma$, and hence $F$ is well-defined on $\GamG\times(0,1)$.
\end{proof}

Note that for $R\in\SO(d)$ we have
\begin{equation}\label{Kinv}
F_{\scrD R}\bigg(M \begin{pmatrix} 1 & 0 \\ 0 & R \end{pmatrix},t\bigg) = F_\scrD(M ,t) .
\end{equation}

Let us define
\begin{equation}
\scrQ_\scrD(M) = \bigcup_{t\in(0,1)}\scrQ_\scrD(M,t) = \big\{ (u,\vec v)\in\ZZ^{d+1} M\;\big| \; |u|<1, \;\vec v \in \RR_{>0}\scrD \big\} .
\end{equation}
The set 
\begin{equation}\label{MD10}
\scrM_\scrD(M)= \big\{| \vec v | \;\big|\;  (u,\vec v)\in\scrQ_\scrD(M) \big\}
\end{equation}
contains the set of values taken by the function $t\mapsto F_\scrD(M,t)$. It is a locally finite subset of $\RR_{>0}$, i.e., there are at most finitely many points in any bounded interval.
It follows that for fixed $M$, the function $t\mapsto F_\scrD(M,t)$ is piecewise constant.

We denote by 
\begin{equation}
\scrG_\scrD(M)=|\{ F_\scrD(M,t) \mid 0<t<1\}|
\end{equation}
the number of distinct values attained by the function $t\mapsto F_\scrD(M,t)$. For $N>0$, let 
\begin{equation}
\scrG_{\scrD,N}(M)=|\{ F_\scrD(M,\tfrac{n}{N_+}) \mid 1\leq n \leq N\}|.
\end{equation}
We have $\scrG_{\scrD,N}(M)\leq \scrG_\scrD(M)$,
and so in particular
\begin{equation}\label{keyineq}
g_N(\scrD) = \scrG_{\scrD,N}(A_{N_+}(\vec\alpha)) \leq \scrG_\scrD(A_{N_+}(\vec\alpha)).
\end{equation}

\section{Geometric lemmas in dimension $d=2$}\label{sec.2Dprep}

To fix notation for our subsequent discussion, we define a representative set of vectors $(u_i,\vec v_i)\in\scrQ_\scrD(M)$, for which the lengths $|\vec v_i|$ are distinct, and each of which corresponds to an element in the set 
\begin{equation}
\scrF_\scrD(M)=\{ F_\scrD(M,t) \mid 0<t<1\}.
\end{equation}
To be specific, for each $M\in\GamG$ we fix vectors $(u_1,\vec{v}_1),\ldots,(u_K,\vec{v}_K)\in\scrQ_\scrD(M)$ with $K=\scrG_\scrD(M)$, so that the following conditions hold:
\begin{enumerate}
	\item[(V1)] $0<|\vec{v}_1|<|\vec{v}_2|<\cdots <|\vec{v}_K|$.
	\item[(V2)] For each $\delta\in\scrF_\scrD(M)$
	there exists an $1\le i\le K$ such that $\delta=|\vec{v}_i|$.
	\item[(V3)] For each $1\le i\le K$, there exists a $t\in(0,1)$ such that $(u_i,\vec{v}_i)\in\scrQ_\scrD(M,t)$ and $|\vec v_i|=F_\scrD(M,t)$.
\end{enumerate}

Let us now focus on the case $d=2$. In the following we identify the unit circle $\SS_1^1$ with the interval $[-\pi,\pi)\bmod 2\pi$, so that $0$ corresponds to direction $(1,0)\in\SS_1^1$. In view of the rotation invariance \eqref{Kinv} we may assume without loss of generality that $\scrD\subseteq\SS_1^1$ is centered at $\theta=0$, i.e., $\scrD= [-\tau/2,\tau/2)$, which we still view as a subset of $\SS_1^1$ (not $\RR$).

Our proof of Theorem \ref{thm.2DBoundedGaps}, or rather the proofs of the more general Theorems \ref{thm.2DBoundedGapsFULL} and \ref{thm.2DBoundedGapsFULL2} below, will be divided into three main cases, which together cover all possible angles  $\tau\in(\pi,2\pi]$ described in \eqref{eqn.gNUpperBd}. In each of these cases we will partition $\scrD$ into subsets, consisting of a {\em symmetric} set $\scrS\subset\scrD$ (symmetric with respect to the rotation $\theta \mapsto \theta+\pi\bmod 2\pi$) and up to three {\em asymmetric} subsets. For notational convenience, let us set \begin{equation}
\psi=2\pi-\tau\in [0,\pi) \quad\text{and}\quad \phi=\tau-\pi\in(0,\pi].
\end{equation}
First we specify our definitions of the asymmetric subsets in each of the three main cases.\vspace*{14bp}

\noindent{\textbf{Case (C1):}} If $5\pi/3\le\tau\le 2\pi$ then we define one asymmetric subset $\mc{A}_0\subset\scrD$ by
\begin{equation}
\mc{A}_{0}= [-\psi/2,\psi/2).
\end{equation}
(Note that this is the empty set if $\tau=2\pi$.)\vspace*{10bp}

\noindent{\textbf{Case (C2):}} If $4\pi/3\le\tau< 5\pi/3$ then we define two asymmetric subsets $\mc{A}_{-1}$ and $\mc{A}_1$ by
\begin{equation}
\mc{A}_{-1}=[-\psi/2,0), \qquad \mc{A}_{1}=[0,\psi/2).
\end{equation}

\noindent{\textbf{Case (C3):}} If $\pi<\tau< 4\pi/3$ then we define three asymmetric subsets $\mc{A}_{-1}, \mc{A}_0,$ and $\mc{A}_1$ by
\begin{equation}
\mc{A}_{-1}=[-\psi/2,-\pi/6), \quad
\mc{A}_{0}=[-\pi/6,\pi/6), \quad \mc{A}_{1}=[\pi/6,\psi/2).
\end{equation}

In all three cases, we define the symmetric subset $\mc{S}$ by 
\begin{equation}
\mc{S}=[-\tau/2,-\psi/2)\cup[\psi/2,\tau/2) .
\end{equation}
It is clear that $\mc{S}$ is the largest symmetric subset of $\scrD$, that $\scrD$ is the disjoint union of $\scrS$ and its asymmetric subsets, and that each asymmetric subset is a half-open interval of length at most $\pi/3$.

Now we will establish several propositions which will help streamline the proofs of Theorems \ref{thm.2DBoundedGapsFULL} and \ref{thm.2DBoundedGapsFULL2} below (which in turn will imply Theorem \ref{thm.2DBoundedGaps}). First we will need the following elementary fact which, for future reference, we state for arbitrary dimension $d\ge 2$.
\begin{prop}\label{prop.LengthReduc}
If $d\ge 2$ and if the angle between two non-zero vectors $\vec{w}_1,\vec{w}_2\in\R^d$ is less than $\pi/3$, then
\begin{equation}
|\vec{w}_1-\vec{w}_2|<\max\left\{|\vec{w}_1|,|\vec{w}_2|\right\}.
\end{equation}
Furthermore, this inequality also holds if the angle between $\vec{w}_1,\vec{w}_2\in\R^d$ is equal to $\pi/3$, as long as $|\vec{w}_1|\not=|\vec{w}_2|$.
\end{prop}
\begin{proof}
For the first part of the proposition, suppose without loss of generality that $|\vec{w}_1|\le|\vec{w}_2|.$ Then the vectors $\vec{w}_1/|\vec{w}_2|$ and $\vec{w}_2/|\vec{w}_2|$ lie in the closed unit ball and are separated by an angle less than $\pi/3$. Therefore
\begin{equation}
\left|\frac{\vec{w}_1}{|\vec{w}_2|}-\frac{\vec{w}_2}{|\vec{w}_2|}\right|<1,
\end{equation}
and the result follows. Furthermore, under the assumptions of the second part of the proposition, we draw the same conclusion.
\end{proof}
Next we will prove several propositions which place various restrictions on the integer $K=\scrG_\scrD(M)$ defined at the start of this section; recall (V1)--(V3).

\begin{prop}\label{prop.SymSmallT}
If for some integer $1\le i\le K$, we have that $u_i\in (-1/2,1/2)$ and $\vec{v}_i\in\RR_{>0}\mc{S}$, then we must have that $i=K$.
\end{prop}
\begin{proof}
	Suppose first that $u_i\in[0,1/2)$. Then for any $0<t<1-u_i$, we have that $(u_i,\vec v_i)\in\scrQ_\scrD(M,t)$ and thus $|\vec{v}_i|\geq F_\scrD(M,t)$. By the symmetry of $\scrS$ we have $(-u_i,-\vec v_i)\in\scrQ_\scrD(M)$. Thus for any $u_i<t<1$, we have that $(-u_i,-\vec v_i)\in\scrQ_\scrD(M,t)$ and thus $|\vec{v}_i|\geq F_\scrD(M,t)$. Since $u_i<1/2$, we conclude that $|\vec{v}_i|\geq F_\scrD(M,t)$ for all $0<t<1$, which proves $i=K$. The case $u_i\in(-1/2,0]$ follows from the same argument.
\end{proof}

\begin{prop}\label{prop.v_KSmallT}
If for some $i$ and $j$ with $1\le i,j\le K,$ we have that $u_i\in (-1/2,0]$ and $u_j\in [0,1/2)$, then $i=K$ or $j=K$.
\end{prop}
\begin{proof}
Under the hypotheses of the proposition, we have $|\vec{v}_i|\geq F_\scrD(M,t)$ for $-u_i<t<1$ and $|\vec{v}_j|\geq F_\scrD(M,t)$ for $0<t<1-u_j$. This covers all possible values of $t\in(0,1)$, and shows that $i$ or $j$ must equal $K$.
\end{proof}

\begin{prop}\label{prop.NoSameTVal}
	If $1\le i,j\le K$ and $u_i=u_j$, then $i=j$.
\end{prop}
\begin{proof}
Suppose by way of contradiction that $i\not= j$, and without loss of generality that $i<j$. Then for any $t\in (0,1)$ satisfying $-t<u_i=u_j\le 1-t$, we would have by (V1) that 
\begin{equation}
F_\scrD(M,t) \le |\vec{v}_i|<|\vec{v}_j|.
\end{equation}
However this contradicts condition (V3), so we must have that $i=j$.
\end{proof}

\begin{prop}\label{prop.ADDI2} Let $1\le i\le K$ and $(u,\vec v)\in\scrQ_\scrD(M)$. 
If
	\begin{equation}-1<u_i\leq u\leq 0 \qquad\text{or}\qquad 0\leq u\leq u_i<1\end{equation}
then 	$|\vec{v}_i|\leq |\vec{v}|$.
\end{prop}

\begin{proof}
Suppose that $0\leq u\leq u_i<1$; the other case follows by symmetry. By (V3) there exist $t_i\in(0,1-u_i)$ such that $F_\scrD(M,t_i) = |\vec{v}_i|$. Furthermore $F_\scrD(M,t) \le |\vec{v}|$ for all $t\in (0,1-u)$. Thus, taking $t=t_i\in(0,1-u)$, we have $|\vec{v}_i|\leq |\vec{v}|$.
\end{proof}

\begin{prop}\label{prop.ADDI} Let $1\le i,j\le K$. If
	\begin{equation}-1<u_i<u_j\leq 0 \qquad\text{or}\qquad 0\leq u_j<u_i<1\end{equation}
then 	$|\vec{v}_i|<|\vec{v}_j|$ and $i<j$.
\end{prop}
\begin{proof}
In view of (V1), this is a direct consequence of Proposition \ref{prop.ADDI2}.
\end{proof}

\begin{prop}\label{prop.LargeTSmallAngle2}
	Let $\tau\geq \pi$, $1\le j\le K$ and $(u,\vec{v})\in\scrQ_\scrD(M)$ 
	such that $\vec{v}_j\neq \vec{v}$ 
	and 
	$0<|\vec{v}|\le|\vec{v}_j|$. 
	Suppose the angle between the vectors $\vec{v}_j$ and $\vec{v}$ is less than $\pi/3$. If 
	\begin{equation}-1<u\leq u_j\leq -1/2 \qquad\text{or}\qquad 1/2\leq u_j\leq u<1,\end{equation}
	then $j=K$.
\end{prop}
\begin{proof}
	We consider the case $1/2\leq u_j\leq u<1$; the proof for the alternative follows from the same argument by symmetry. By the assumption on the angle between the vectors $\vec{v}_j$ and $\vec{v}$, we have by Proposition \ref{prop.LengthReduc} that
	$|\vec{v}_j-\vec{v}|<|\vec{v}_j|$. Since $\vec{v}_j\neq \vec{v}$ and $\tau\geq\pi$, we have that at least one of $\vec{v}_j-\vec{v}$, $\vec{v}-\vec{v}_j$ is in $\RR_{>0}\scrD$.
	
	First suppose that $\vec{v}-\vec{v}_j\in\RR_{>0}\scrD$. Then, since 
	$$0\leq u-u_j <1/2\leq u_j<1, $$ 
	we have $(u-u_j,\vec v-\vec v_j)\in\scrQ_\scrD(M)$ and by Proposition \ref{prop.ADDI2} that $|\vec{v}-\vec{v}_j|\geq |\vec{v}_j|$. This is a contradiction, so we conclude that $\vec{v}-\vec{v}_j\notin\RR_{>0}\scrD$. 
	
	The only other possibility is that $\vec{v}_j-\vec{v}\in\RR_{>0}\scrD$. In this case, $(u_j-u,\vec v_j-\vec v)\in\scrQ_\scrD(M),$
	\begin{equation}
	u_j-u\leq 0\quad\text{and}\quad u-u_j<1-u_j .
	\end{equation}
	It follows from this that, for $u-u_j< t<1$, we have $F_\scrD(M,t)\leq |\vec{v}-\vec{v}_j|<  |\vec{v}_j|$ and for $0< t<1-u_j$ (which in particular holds for all $t$ with $0 < t \leq u-u_j$), we have $F_\scrD(M,t)\leq |\vec{v}_j|$. Therefore $j=K$.
\end{proof}

\begin{prop}\label{prop.LargeTSmallAngle} Let $\tau\geq \pi$ and $1\le i,j\le K$. Suppose the angle between the vectors $\vec{v}_i$ and $\vec{v}_j$ is less than $\pi/3$. If 
	\begin{equation}-1<u_i<u_j\leq -1/2 \qquad\text{or}\qquad 1/2\leq u_j<u_i<1,\end{equation}
	then $j=K$.
\end{prop}

\begin{proof}
This is a direct corollary of Proposition \ref{prop.LargeTSmallAngle2} (take $u=u_i$). 
\end{proof}

\begin{prop}\label{prop.SmallTSmallAngle} Let $\tau\geq \pi$ and $1\le i,j\le K$. Suppose the angle between the vectors $\vec{v}_i$ and $\vec{v}_j$ is less than $\pi/3$. If
	\begin{equation}-1/2<u_i<u_j\leq 0 \qquad\text{or}\qquad 0\leq u_j<u_i<1/2,\end{equation}
	then $\vec{v}_j-\vec{v}_i\notin\RR_{>0}\scrD$ and
		\begin{equation}
		|\vec{v}_i|\le |\vec{v}_i-\vec{v}_j|<|\vec{v}_j|.
		\end{equation}
\end{prop}
\begin{proof}
	We assume $0\leq u_j<u_i<1/2$; the other case follows by symmetry.
	It follows from Proposition \ref{prop.ADDI} that $i<j$ and $|\vec{v}_i|<|\vec{v}_j|$, and it follows from Proposition \ref{prop.LengthReduc} that
	$|\vec{v}_i-\vec{v}_j|<|\vec{v}_j|$.
	
	Suppose, contrary to what we are trying to prove, that $\vec{v}_j-\vec{v}_i\in\RR_{>0}\scrD$. Then, since $0<u_i-u_j<1/2$, we have $(u_j-u_i,\vec v_j-\vec v_i)\in\scrQ_\scrD(M)$ and hence for $u_i-u_j<t<1$ we have $F_\scrD(M,t)\leq |\vec{v}_i-\vec{v}_j| <|\vec{v}_j|$.
	Furthermore, for $0<t<1-u_i$ we have $F_\scrD(M,t)\leq |\vec v_i|<|\vec v_j|$. Now $u_i-u_j<1-u_i$, so $F_\scrD(M,t)<|\vec v_j|$ for all $t\in(0,1)$.
	But by (V3) there is $t\in(0,1)$ such that $F_\scrD(M,t)= |\vec v_j|$. This is a contradiction, so we conclude that $\vec{v}_j-\vec{v}_i\notin\RR_{>0}\scrD$. 
	
	It remains to show that $|\vec{v}_i|\le |\vec{v}_i-\vec{v}_j|$. Since $\vec{v}_j-\vec{v}_i\notin\RR_{>0}\scrD$ and $\tau\geq \pi$ we have $\vec{v}_i-\vec{v}_j\in\RR_{>0}\scrD$. Then $(u_i-u_j,\vec v_i-\vec v_j)\in\scrQ_\scrD(M)$ with $0<u_i-u_j<1/2$. Thus for $0<t<1-(u_i-u_j)$ we have $F_\scrD(M,t)\leq |\vec{v}_i-\vec{v}_j|$. Note that by (V3) there exists a $t_i\in(0,1-u_i)$ such that $F_\scrD(M,t_i)=|v_i|$. Now $1-u_i\leq 1-(u_i-u_j)$ and so $F_\scrD(M,t_i)=|v_i|\leq |\vec{v}_i-\vec{v}_j|$ as needed.
\end{proof}

The previous proposition will be used in the proof of Theorem \ref{thm.2DBoundedGapsFULL} in conjunction with the following two elementary geometric propositions.
\begin{prop}\label{prop.AsymCone1}
	Suppose that 
	\begin{itemize}
\item we are in case (C1) or (C3) and $\vec{w}_1,\vec{w}_2\in\RR_{>0}\mc{A}_0$, or 
\item we are in case (C2) and $\vec{w}_1,\vec{w}_2\in\RR_{>0}\mc{A}_{-1}$, or
\item we are in case (C2) and $\vec{w}_1,\vec{w}_2\in\RR_{>0}\mc{A}_{1}$.
\end{itemize}
If 
	\begin{equation}\label{eqn.AsymConeProp1}
	|\vec{w}_1|\le |\vec{w}_1-\vec{w}_2|<|\vec{w}_2|,
	\end{equation}
	then $\vec{w}_2-\vec{w}_1\in\RR_{>0}\scrD$.
\end{prop}

\begin{proof}
	First suppose that we are in case (C2), that $\vec{w}_1$ and $\vec{w}_2$ point in direction $\mc{A}_1,$ and that \eqref{eqn.AsymConeProp1} holds. We will argue using Figure \ref{fig.ConeFig2}.

	\begin{figure}[ht]
		\centering
		\def\svgwidth{0.6\columnwidth}
		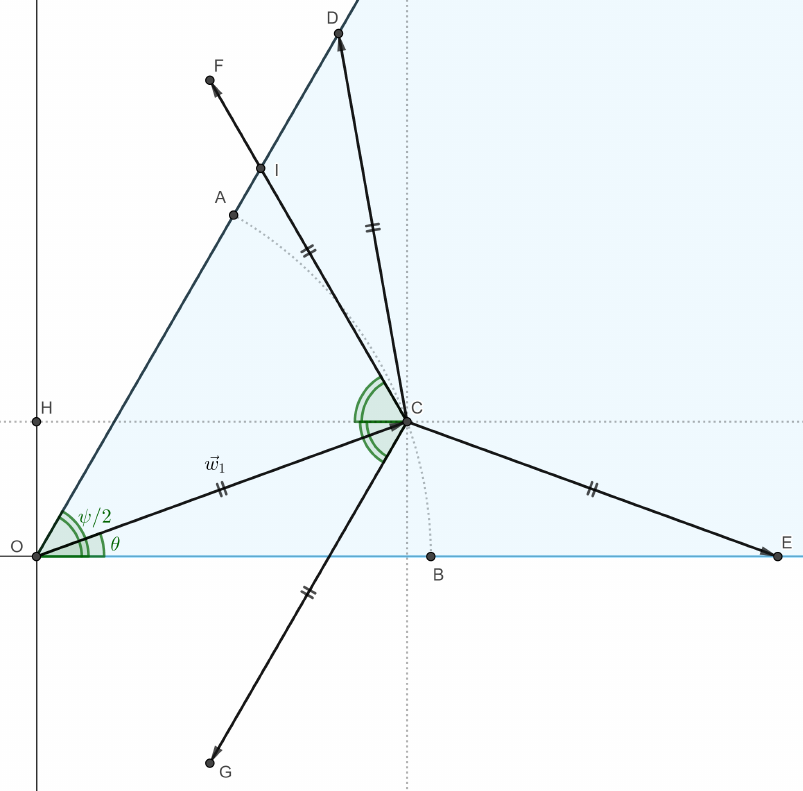
		\caption{Diagram of $\mc{A}_1$ in case (C2)}\label{fig.ConeFig2}
	\end{figure}
	
	 In the figure, angle AOB measures $\psi/2$ and sweeps out $\mc{A}_1$, and the line through H and C is parallel to the $x$-axis. The vector $\vec{w}_1$ is shown, and the angles HCF and HCG also measure $\psi/2$. Also, we have labeled the angle between $\vec{w}_1$ and the positive real axis as $\theta$ (not to be confused with other uses of $\theta$ outside the scope of this proof).
	 
	 By condition \eqref{eqn.AsymConeProp1}, vector $\vec{w}_2$ has to lie outside of both the circle of radius $|\vec{w}_1|$ centered at O, and the circle of radius $|\vec{w}_1|$ centered at C. The circle of radius $|\vec{w}_1|$ centered at C intersects the boundary of $\RR_{>0}\mc{A}_1$ at the three points D, O, and E, and the points F and G have also been chosen so that they lie on this circle. We will show (as indicated in the figure) that F and G lie outside of $\RR_{>0}\mc{A}_1$. This will complete the proof in this sub-case since, if $\vec{w}_2-\vec{w}_1$ were not in $\RR_{>0}\scrD$ then $\vec{w}_2$ would have to lie in the cone swept out by angle FCG, above the ray originating from C and passing through G, and on or below the ray originating from C and passing through F. This, together with the condition that it lies outside of the circle of radius $|\vec{w}_1|$ centered at C, would force it to lie outside of $\RR_{>0}\mc{A}_1$, which is contradictory to our hypotheses.
	 
	 It is clear from the fact the $\psi/2\le\pi/3$ that G lies below the $x$-axis, so it cannot be in $\RR_{>0}\mc{A}_1$ (in fact we only need $\psi/2<\pi/2$ for this to hold). To see why F is not in $\RR_{>0}\mc{A}_1$, first note that angle OCH has measure $\theta$, from which it follows that angle CIO has measure $\pi-\psi$ (to avoid circular reasoning, the point I is defined as the intersection of the line through C and F with the line through O and A). On the other hand, angle CDO has measure $\psi/2-\theta$ and, since $3\psi/2\le\pi$, we have that
	 \begin{equation}
	\psi/2-\theta\le \pi-\psi.
	 \end{equation}
	 This implies that the point F lies on or to the left of the line through O and D, therefore it is not in $\RR_{>0}\mc{A}_1$. The proof for case (C2) when $\vec{w}_1$ and $\vec{w}_2$ point in direction $\mc{A}_{-1}$ follows by symmetry.
	
	Next suppose that we are in case (C1), that $\vec{w}_1$ and $\vec{w}_2$ point in direction $\mc{A}_0,$ and that \eqref{eqn.AsymConeProp1} holds. Here the proof is similar, and we will argue using Figure \ref{fig.ConeFig1}. Once again, let $\theta$ denote the angle between $\vec{w}_1$ and the positive real axis.

\begin{figure}[ht]
	\centering
	\def\svgwidth{0.6\columnwidth}
	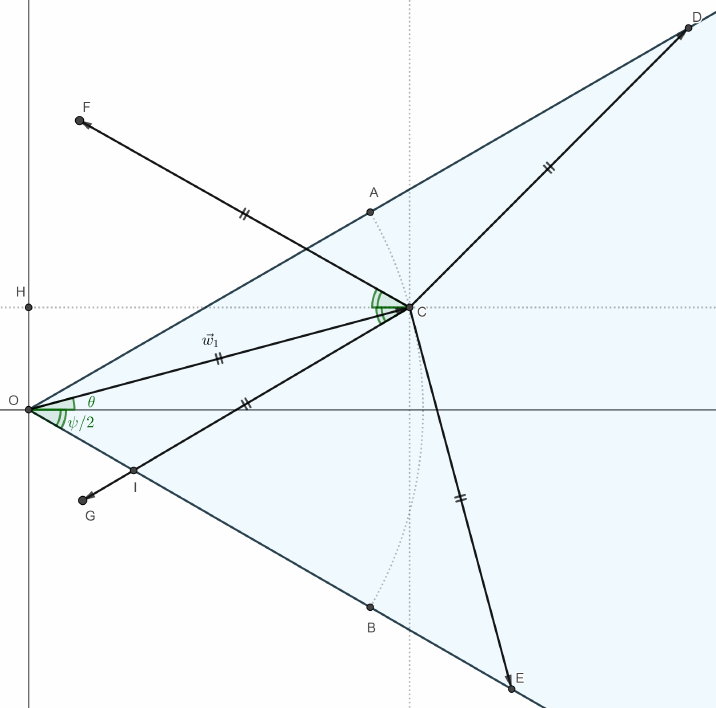
	\caption{Diagram of $\mc{A}_0$ in case (C1)}\label{fig.ConeFig1}
\end{figure}

First assume that $0\le\theta<\psi/2$. Since $\psi/2\le \pi/6$, the argument given above implies again that F lies outside of $\RR_{>0}\mc{A}_0$. Angle HCO has measure $\theta$, therefore angle OCI has measure $\psi/2-\theta$, and it follows that angle CIO has measure $\pi-\psi$. Since triangle OCE is isosceles, angle CEO has measure $\psi/2+\theta$, and since
\begin{equation}
\psi/2+\theta<\pi-\psi,
\end{equation}
this implies that G lies outside of $\RR_{>0}\mc{A}_0$. This argument actually works for all $\theta$ and $\psi$ satisfying $0\le \theta<\psi/2\le\pi/4,$ and a symmetrical argument applies when $-\pi/4\le-\psi/2<\theta<0$. The proof for case (C3), when $\vec{w}_1$ and $\vec{w}_2$ point in direction $\mc{A}_0,$ follows from the same argument.

\end{proof}

\begin{prop}\label{prop.AsymCone2}
	Suppose that we are in case (C3) and that $\vec{w}_1,\vec{w}_2,\ldots,\vec{w}_n$ are any vectors which all point in direction $\mc{A}_{-1}$, or which all point in direction $\mc{A}_1$. If, for each $1\le i\le n-1$, we have that $\vec{w}_{i+1}-\vec{w}_i\not\in\RR_{>0}\scrD$ and that
	\begin{equation}\label{eqn.AsymConeProp2}
	|\vec{w}_i|\le |\vec{w}_i-\vec{w}_{i+1}|<|\vec{w}_{i+1}|,
	\end{equation}
	then we must have that
	\begin{equation}
	n \le 1+\left\lfloor\frac{\sin(\tau/2+\pi/6)}{\sin(\tau-\pi)}\right\rfloor .
	\end{equation}
\end{prop}

\begin{proof}
	Suppose $\vec{w}_1,\vec{w}_2,\ldots,\vec{w}_n$ point in direction $\mc{A}_{1}$ and consider Figure \ref{fig.ConeFig3}.
	
\begin{figure}[ht]
	\centering
	\def\svgwidth{0.4\columnwidth}
	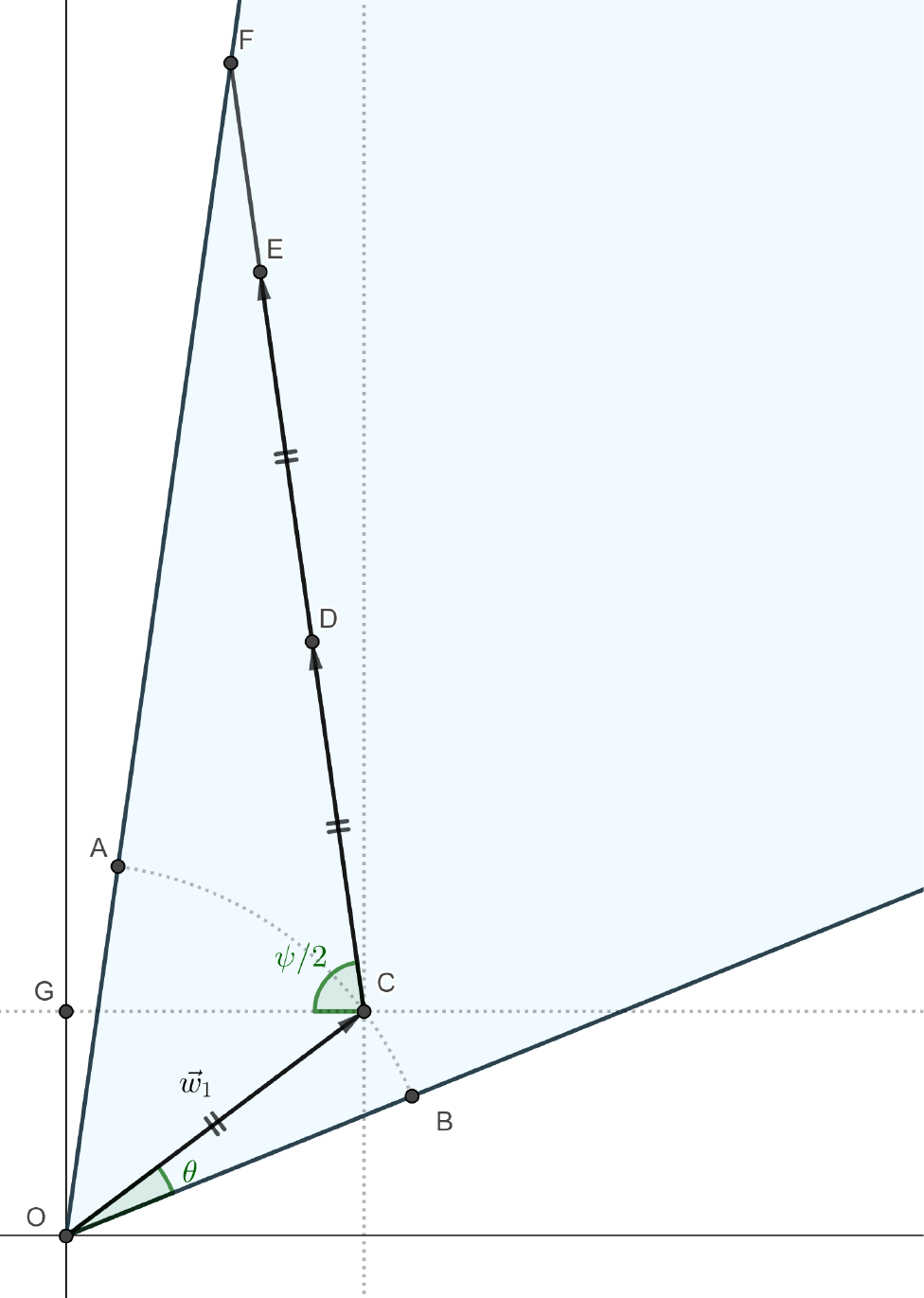
	\caption{Diagram of $\mc{A}_1$ in case (C3)}\label{fig.ConeFig3}
\end{figure}

In the figure, angle AOB measures $\psi/2-\pi/6<\pi/3$ and sweeps out $\mc{A}_1$, the vector $\vec{w}_1$ has initial point O and terminal point C, and the line through G and C is parallel to the $x$-axis. For each $1\le i\le n-1$, we have that $\vec{w}_{i+1}-\vec{w}_i\not\in\RR_{>0}\scrD$ and
\begin{equation}
|\vec{w}_{i+1}-\vec{w}_i|\ge |\vec{w}_1|,
\end{equation}
so it follows that
\begin{equation}\label{eqn.AsymConeNBd}
n-1\le \left\lfloor\frac{|\text{CF}|}{|\vec{w}_1|}\right\rfloor.
\end{equation}
Line segment CF is longest when $\theta=0$, which (using the law of sines) gives the bound 
\begin{equation}
|\text{CF}|\le \frac{|\vec{w}_1|\sin (\psi/2-\pi/6)}{\sin(\pi-\psi)}.
\end{equation}
Substituting $\psi=2\pi-\tau$ gives
\begin{equation}
\frac{|\text{CF}|}{|\vec w_1|}\le 
\frac{\sin (\tau/2+\pi/6)}{\sin(\tau-\pi)}.
\end{equation}
Combining this with \eqref{eqn.AsymConeNBd} completes the proof of the proposition.
\end{proof}

Finally, to obtain the bounds reported in some of the cases of Theorem \ref{thm.2DBoundedGaps}, we will need to gather together a few more facts. The following proposition is an extension of Proposition \ref{prop.ADDI2}, for the special case when one of the vectors in the hypotheses lies in the direction determined by the symmetric set $\mc{S}$.

\begin{prop}\label{prop.SymConeOrder}
Let $1\le i,j\le K$, $i\neq j$, and let $\vec{v}_j\in\RR_{>0}\mc{S}$. If
\begin{equation}-1< u_i\leq -u_j\leq 0 \quad\text{or}\quad 0\leq -u_j\leq u_i<1, \end{equation}
then 	$|\vec{v}_i|<|\vec{v}_j|$ and $i<j$.
\end{prop}

\begin{proof}
The vector $-\vec{v}_j$ is in $\mc{S}$, and hence $(-u_j,-\vec{v}_j)\in\scrQ_\scrD(M)$. Proposition \ref{prop.ADDI2} then yields the statement.
\end{proof}

\begin{prop}\label{prop.SymConeSmallAngle}
Let $1\le i,j\le K$, and let $\vec{v}_j\in\RR_{>0}\mc{S}$. Assume the angle between $\vec{v}_i$ and $-\vec{v}_j$ is less than $\pi/3$. If
	\begin{equation}-1< u_i \leq -u_j \leq -1/2 \quad\text{or}\quad 1/2\leq -u_j \leq u_i< 1,\end{equation}
	then $j=K$.
\end{prop}

\begin{proof}
The proof is similar to that of Proposition \ref{prop.LargeTSmallAngle2}, which would directly apply if we had assumed $\vec{v}_i\in\RR_{>0}\mc{S}$ rather than $\vec{v}_j\in\RR_{>0}\mc{S}$. 
	
	The assumption on the angle implies $i\neq j$. We consider the case $1/2\leq -u_j\leq u_i<1$; the proof for the alternative follows from the same argument by symmetry. By the assumption on the angle between the vectors $\vec{v}_i$ and $-\vec{v}_j$, we have by Proposition \ref{prop.LengthReduc} that
	$|\vec{v}_i+\vec{v}_j|<|\vec{v}_j|$. Since $\vec{v}_i\neq \vec{v}_j$  and $\tau\geq\pi$, we have that at least one of $\pm(\vec{v}_i+\vec{v}_j)$ is in $\RR_{>0}\scrD$. Suppose $\vec{v}_i+\vec{v}_j\in\RR_{>0}\scrD$. 
	
	Then, since $-(u_j,\vec{v}_j)\in\scrQ_\scrD(M)$ and
	$$0\leq u_i+u_j < 1/2\leq -u_j<1, $$ 
	we have $(u_i+u_j,\vec v_i+\vec v_j)\in\scrQ_\scrD(M)$ and by Proposition \ref{prop.ADDI2} that $|\vec{v}_i+\vec{v}_j|\geq |\vec{v}_j|$, a contradiction. Therefore $\vec{v}_i+\vec{v}_j\notin\RR_{>0}\scrD$ and we must have $-(\vec{v}_i+\vec{v}_j)\in\RR_{>0}\scrD$. This means $-(u_i+u_j,\vec v_i+\vec v_j)\in\scrQ_\scrD(M)$ and 
	\begin{equation}
	0 \leq  u_i+u_j<1+u_j .
	\end{equation}
	It follows from this that, for all $t$ with $u_i+u_j< t<1$, we have $F_\scrD(M,t)\leq |\vec{v_i}+\vec{v}_j|<  |\vec{v}_j|$ and for $0< t<1+u_j$ (which in particular holds for all $t$ with $0 < t \leq u_i+u_j$), we have $F_\scrD(M,t)\leq |\vec{v}_j|$. Therefore $j=K$.
\end{proof}

The previous proposition allows us to deduce the following simple and useful result. Recall that $\phi=\tau-\pi$, so that $\tau=\psi+2\phi$.
\begin{prop}\label{prop.SymConeUB}
Let $S_2$ denote the number of integers $i$ with $1\le i\le K,~\vec{v}_i\in\RR_{>0}\mc{S}$, and $u_i\in (-1,-1/2]\cup [1/2,1)$. Then
\begin{equation}\label{eqn.SymVecsUB}
S_2\le 
\begin{cases}
1+2\left\lceil\frac{\phi}{\pi/3}\right\rceil & \text{if $\vec{v}_K\in\RR_{>0}\mc{S}$ and $u_K\in (-1,-1/2]\cup [1/2,1)$,}\\[10pt]
2\left\lceil\frac{\phi}{\pi/3}\right\rceil & \text{otherwise.}\\
\end{cases}
\end{equation}
\end{prop}
\begin{proof}
The quantity $\phi$ is the angle swept out by the part of $\mc{S}$ which lies above the $x$-axis. The maximum number of vectors which can be placed in this region, so that the angles between any two vectors is at least $\pi/3$, is $\lceil\phi/(\pi/3)\rceil$. The upper bound in \eqref{eqn.SymVecsUB} therefore follows from combining the results of Propositions \ref{prop.LargeTSmallAngle} and \ref{prop.SymConeSmallAngle}.
\end{proof}

\section{Explicit upper bounds in dimension $d=2$, part 1}\label{sec.2DBddPf}

Throughout this section we take $d=2$. In view of \eqref{keyineq}, the following statement directly implies all cases of Theorem \ref{thm.2DBoundedGaps}, except the case when $\scrD=\SS_1^1$ (which is handled in the next section).

\begin{thm}\label{thm.2DBoundedGapsFULL}
Let $d=2$, and assume $\scrD\subseteq\SS_1^1$ is a half-open interval of arclength $\tau>\pi$.
Then for any $M\in \SL(3,\RR)$ we have that
\begin{equation}\label{eqn.gNUpperBdFULL}
\scrG_\scrD(M)\le\begin{cases}
9&\text{if}\quad5\pi/3<\tau<2\pi,\\
8&\text{if}\quad\tau=5\pi/3,\\
9&\text{if}\quad4\pi/3<\tau<5\pi/3,\\
8&\text{if}\quad\tau=4\pi/3,\\
12+2\left\lfloor\frac{\sin (\tau/2+\pi/6)}{\sin(\tau-\pi)}\right\rfloor&\text{if}\quad\pi<\tau<4\pi/3.
\end{cases}
\end{equation}
\end{thm}
The remainder of this section is dedicated to the proof of this theorem. For the proof, we will apply the propositions from the previous section to each of the five cases described in \eqref{eqn.gNUpperBdFULL}. To summarize the main points of our arguments:\vspace*{1bp}
\begin{enumerate}[(i)]
	\item Let $S_1$ denote the number of $1\le i\le K$ with $u_i\in (-1/2,1/2)$ and $\vec{v}_i\in\RR_{>0}\mc{S}$. Proposition \ref{prop.SymSmallT} guarantees that $S_1\le 1$, and that if $S_1=1$ then the corresponding value of $i$ equals $K$.\vspace*{10bp}
	\item Let $S_2$ denote the number of $1\le i\le K$ with $u_i\in (-1,-1/2]\cup [1/2,1)$ and $\vec{v}_i\in\RR_{>0}\mc{S}$. Proposition \ref{prop.SymConeUB} gives an upper bound for $S_2$.\vspace*{10bp}
	\item Let $A_1$ denote the number of $1\le i\le K$ with $u_i\in (-1/2,1/2)$ and with $\vec{v}_i$ in direction of any asymmetric subset. Propositions \ref{prop.v_KSmallT} and \ref{prop.SmallTSmallAngle}-\ref{prop.AsymCone2} give upper bounds for $A_1+S_1$.\vspace*{10bp}
	\item Let $A_2$ denote the number of $1\le i\le K$ with $u_i\in (-1,-1/2]\cup [1/2,1)$ and with $\vec{v}_i$ in direction of any asymmetric subset. Propositions \ref{prop.NoSameTVal}-\ref{prop.LargeTSmallAngle} give upper bounds for $A_2+S_2$.\vspace*{1bp}	
\end{enumerate}
In all cases, we have that $\scrG_\scrD(M)=K=S_1+A_1+S_2+A_2.$ In what follows, recall that $\psi=2\pi-\tau$ and $\phi=\tau-\pi$.\vspace*{10bp}

\noindent \textbf{Case (C1)}, $\tau=2\pi$: This is actually not one of the cases considered in Theorem \ref{thm.2DBoundedGapsFULL}, but we include it as a demonstration of the proof technique, and to provide an easy argument that $\scrG_\scrD(M)\le 6$. This bound will be improved in the next section, by a slightly more complicated argument, to show that $\scrG_\scrD(M)\le 5$.

In this case $\psi=0$ and $\phi=\pi$, so it is clear that $A_1=A_2=0$. We claim that $S_2\le 5$. Suppose by way of contradiction that $S_2\ge 6$. Since $\tau=2\pi$, we may assume without loss of generality that the vectors $(u_i,\vec{v}_i)$ have been chosen so that $u_i\in [0,1)$ for each $i$. Then we can find $1\le i<j\le K$ with $1/2\le u_i,u_j<1$, and for which the angle between $\vec{v}_i$ and $\vec{v}_j$ is less than or equal to $\pi/3$. By the second part of Proposition \ref{prop.LengthReduc}, we then have that 
\begin{equation}
|\vec{v}_i-\vec{v}_j|<|\vec{v}_j|.
\end{equation}
However, since $|u_i-u_j|<1/2$ and $\vec{v}_i-\vec{v}_j\in\RR_{>0}\mc{S}$, this implies that
\begin{equation}
F_\mc{D}(M,t)\le|\vec{v}_i-\vec{v}_j|<|\vec{v}_j|,
\end{equation}
for all $t\in(0,1)$, which contradicts assumption (V3). Therefore $S_2\le 5$. Since $S_1\le 1,$ this gives the bound $\scrG_\scrD(M)\le 6$.
\vspace*{10bp}

\noindent \textbf{Case (C1)}, $5\pi/3<\tau<2\pi$: In this case $0<\psi<\pi/3$ and $2\pi/3<\phi<\pi$. There is only one asymmetric subset, $\mc{A}_{0}= [-\psi/2,\psi/2).$\vspace*{5bp}

\begin{enumerate}[A.]
	\item Assume $S_1=1$. Then by Proposition \ref{prop.SymSmallT} we must have that
	\begin{equation}\label{eqn.u_KSmall}
	u_K\in (-1/2,1/2).
	\end{equation}
	Proposition \ref{prop.v_KSmallT} implies that we cannot have $i$ and $j$ with $1\le i,j<K,~u_i\in (-1/2,0],~$ and $u_j\in[0,1/2).$ Therefore, by Propositions \ref{prop.SmallTSmallAngle} and \ref{prop.AsymCone1}, we have that $A_1\le 1$. \vspace*{5bp}
	
	Since $\lceil\frac{\phi}{\pi/3}\rceil= 3$ and \eqref{eqn.u_KSmall} holds, Proposition \ref{prop.SymConeUB} implies that $S_2\le 6.$ We claim that, in this sub-case, $S_2+A_2\le 7.$ It is clear from Proposition \ref{prop.LargeTSmallAngle} that, since \eqref{eqn.u_KSmall} holds, we must have that $A_2\le 2$, and that if $A_2=2$ then one element has its first component in $(-1,-1/2]$ and the other in $[1/2,1)$. In order to establish our claim we only need to consider what happens when $S_2\ge 5$. If $S_2\ge 5$ then there must be at least three values of $1\le i< K$ with $u_i\in(-1,-1/2]$ and $\vec{v}_i\in\RR_{>0}\mc{S}$, or at least three values with $u_i\in [1/2,1)$ and $\vec{v}_i\in\RR_{>0}\mc{S}$ (or possibly both, if $S_2=6$). The argument in both cases is the same, so suppose without loss of generality that the former condition holds. Then, there are at least $2$ values $1\le i<j< K$ with $u_i,u_j\in(-1,-1/2],~\vec{v}_i,\vec{v}_j\in\RR_{>0}\mc{S},$ and with $\vec{v}_i$ and $\vec{v}_j$ either both above or both below the $x$-axis. If there were also a value of $1\le k\le K$ with $u_k\in (-1,-1/2]$ and $\vec{v}_k\in\mc{A}_0$ then, since $\phi+\psi=\pi$, at least one pair of the vectors $\vec{v}_i,\vec{v}_j,$ and $\vec{v}_k$ would be separated by an angle of less than $\pi/3$. This, together with Proposition \ref{prop.LargeTSmallAngle}, would contradict \eqref{eqn.u_KSmall}, therefore such a $k$ cannot exist. This means $A_2\le 1,$ and $S_2+A_2\le 6+1=7$.\vspace*{5bp}
	
	\noindent Combining the above bounds gives that $\scrG_\scrD(M)\le 1+1+6+1=9$.\vspace*{5bp}
	
	\item Assume  $S_1=0$. Then by Propositions \ref{prop.SmallTSmallAngle} and \ref{prop.AsymCone1}, we have that $A_1\le 2$. If $A_1=2$ then by Proposition \ref{prop.v_KSmallT} we again have that \eqref{eqn.u_KSmall} holds. The same argument as in the previous Case A yields  $S_2+A_2\le 6+1=7$, and therefore $\scrG_\scrD(M)\le 0+2+6+1 = 9$.\vspace*{5bp}
	
	Assume now $A_1\le 1$. Since $\lceil\frac{\phi}{\pi/3}\rceil= 3$ we have by Proposition \ref{prop.SymConeUB} that $S_2\le 7$, and the case $S_2=7$ can only arise if $\vec{v}_K\in\RR_{>0}\mc{S}$ and $u_K\in (-1,-1/2]\cup [1/2,1)$. \vspace*{5bp}
	\begin{enumerate}
	\item
	Assume first $S_2=7$. Then there exist vectors $\vec{v}_{i_1},\ldots ,\vec{v}_{i_5}\in\RR_{>0}\mc{S}$ with indices 
	\begin{equation}
	1\le i_1<i_2<\cdots<i_5<K
	\end{equation}
	chosen so that the vectors have the smallest possible lengths. There must be at least $2$ of these vectors, say $\vec v_i$ and $\vec v_j$, which both lie above or below the $x$-axis, and with corresponding $u_i$ values either both in $(-1,-1/2]$ or both in $[1/2,1)$. If there were also a value of $1\le k\le K$ with $u_k\in (-1,-1/2]$ and $\vec{v}_k\in\mc{A}_0$ then, since $\phi+\psi=\pi$, at least one pair of the vectors $\vec{v}_i,\vec{v}_j,$ and $\vec{v}_k$ would be separated by an angle of less than $\pi/3$. Proposition \ref{prop.LargeTSmallAngle} implies that one of the indices has to be equal to $K$. Since $i,j<K$, we have $k=K$, contradicting the fact that $\vec{v}_K\in\RR_{>0}\mc{S}$. We have therefore $A_2\le 1$, so $\scrG_\scrD(M)\le 0+1+7+1=9$. 
	\item Assume now $S_2=5$ or $6$. In this case we claim that $A_2\le 2$. To see why this is true, suppose that $A_2\ge 3$. 
	Then, by Proposition \ref{prop.LargeTSmallAngle}, we have $1\le i<j<K$ such that $\vec{v}_i,\vec{v}_j,\vec{v}_K\in\mc{A}_0$ and $u_i,u_j,u_K\notin(-1/2,1/2)$ with $u_i$ and $u_j$ having opposite sign. Since $S_2\geq 5$, there are at least five values of $\ell$ with $1\le \ell<K,~\vec{v}_\ell\in\RR_{>0}\mc{S}$, and $u_\ell\notin(-1/2,1/2)$. By the same arguments as above, at least two of these would have to lie either above or below the $x$-axis, and have $u_\ell$ values both in $(-1,-1/2]$ or both in $[1/2,1)$. This, by Proposition \ref{prop.LargeTSmallAngle}, would then imply that one of $i$ or $j,$ or one of these $\ell$ values, equals $K$, which is a contradiction. Therefore $A_2\le 2$, $S_2\leq 6$ and so $\scrG_\scrD(M)\le 0+1+6+2=9$.  
	\item
	Finally, assume $S_2\le 4$. By Proposition \ref{prop.LargeTSmallAngle} we have that $A_2\le 3$, so  $\scrG_\scrD(M)\le 0+1+4+3=8$.
	\end{enumerate}
\end{enumerate}
\vspace*{10bp}

\noindent \textbf{Case (C1)}, $\tau=5\pi/3$: In this case $\psi=\pi/3$ and $\phi=2\pi/3$, and there is only one asymmetric subset, $\mc{A}_{0}= [-\pi/6,\pi/6)$. In this case the argument is similar to the previous case. The key improvement is in the application of Proposition \ref{prop.SymConeUB}, since now $\lceil\frac{\phi}{\pi/3}\rceil= 2$.\vspace*{5bp}

\begin{enumerate}[A.]
	\item If $S_1=1$ then \eqref{eqn.u_KSmall} holds and we conclude as in the previous Case A that $A_1\le 1,~$ $S_2\le 4$. Proposition \ref{prop.LargeTSmallAngle} and \eqref{eqn.u_KSmall} imply that $A_2\le 2$, and so $\scrG_\scrD(M)\le 1+1+4+2=8$.\vspace*{5bp}
	\item If $S_1=0$ then $A_1\le 2$ as in Case B above. If $A_1=2$ then \eqref{eqn.u_KSmall} holds, so $S_2\le 4$ by Proposition \ref{prop.SymConeUB} and we conclude $\scrG_\scrD(M)\le 0+2+4+2=8$.\vspace*{5bp}
	
	\noindent Now assume $A_1\le 1$. Proposition \ref{prop.SymConeUB} gives $S_2\le 5.$ If $S_2=5$ then, by the argument in Case B (b) above, $A_2\le 2$. We conclude $\scrG_\scrD(M)\le 0+1+5+2=8$. If $S_2\le 4$ then, by the same argument as in Case B (c) above, $A_2\le 3$ and $\scrG_\scrD(M)\le 0+1+4+3=8$.
\end{enumerate}\vspace*{10bp}

\noindent \textbf{Case (C2)}, $4\pi/3<\tau<5\pi/3$: In this case $\pi/3<\psi<2\pi/3$ and $\pi/3<\phi<2\pi/3$. There are two asymmetric subsets, $\mc{A}_{-1}=[-\psi/2,0)$ and $\mc{A}_{1}=[0,\psi/2).$ The argument here is similar to that given above, except that now when we use Propositions \ref{prop.SmallTSmallAngle} and \ref{prop.AsymCone1}, they must be applied to both asymmetric subsets.\vspace*{5bp}

\begin{enumerate}[A.]
	\item  Assume $S_1=1$. Then \eqref{eqn.u_KSmall} holds, and Propositions \ref{prop.SmallTSmallAngle} and \ref{prop.AsymCone1} imply that $A_1\le 2$ (the same argument as in Case A above, now applied to each of the two asymmetric cones). Furthermore $\lceil\frac{\phi}{\pi/3}\rceil= 2$ and thus Proposition \ref{prop.SymConeUB} yields $S_2\le 4$.
	\begin{enumerate}
	\item
	We claim that, if $S_2\ge 3,$ then $A_2\le 2$. Suppose by way of contradiction that $S_2\ge 3$, that $1\le i,j,k<K,~u_i,u_j,u_k\notin (-1/2,1/2),$ and that $\vec{v}_i,\vec{v}_j,\vec{v}_k$ are distinct vectors in direction $\mc{A}_{-1}\cup\mc{A}_1$. It could not be the case that either $u_i,u_j,u_k\in (-1,-1/2]$ or that $u_i,u_j,u_k\in [1/2,1)$, otherwise we would have to have by Proposition \ref{prop.LargeTSmallAngle} that $i,j,$ or $k$ equals $K$. Therefore two of the numbers $u_i,u_j,$ and $u_k$ lie in one of the intervals $(-1,-1/2]$ or $[1/2,1)$, and the other number lies in the other interval. Without loss of generality (the argument is the same in all cases) let us suppose that $u_i,u_j\in (-1,-1/2]$ and that $u_k\in [1/2,1)$. Then there cannot be any values of $1\le \ell<K$ with $u_\ell\in (-1,-1/2]$ and $\vec{v}_\ell\in\RR_{>0}\mc{S}$. If there were then, again since $\psi+\phi=\pi$, at least one pair of the three vectors $\vec{v}_i,\vec{v}_j,$ and $\vec{v}_\ell$ would be separated by an angle of less than $\pi/3$, giving the contradiction that $i,j,$ or $\ell$ equals $K$. This means that there are at least three values of $1\le \ell< K$ for which $u_\ell\in [1/2,1)$ and $\vec{v}_\ell\in\RR_{>0}\mc{S}$. At least two of these vectors lie either above or below the $x$-axis. Since $u_k$ is also in $[1/2,1)$, this means (by the same argument just given) that either $u_k$, or one of these $u_\ell$ values, is $u_K$. This is also a contradiction, so we conclude that $A_1\le 2.$ Putting this all together, we have that $\scrG_\scrD(M)\le 1+2+4+2=9$.
	
	\item If $S_2=1$ or $2$, then we must have that $A_2\le 3$. To see why, suppose $A_2\ge 4$ and choose $1\le i_1<i_2<i_3<i_4< K$ so that $u_{i_1},\ldots,u_{i_4}\notin (-1/2,1/2)$ and $\vec{v}_{i_1},\ldots,\vec{v}_{i_4}\notin\mc{S}$. By the same argument as before, we cannot have three of these $u_i$ values in either $(-1,-1/2]$ or in $[1/2,1)$. Therefore two of them must be in one of these intervals, and the other two must be in the other interval. There is at least one value of $1\le\ell<K$ with $u_\ell\notin(-1/2,1/2)$ and $\vec{v}_\ell\in\RR_{>0}\mc{S}$, and as before, this implies that one of the vectors we have just listed is $\vec{v}_K$; a contradiction with \eqref{eqn.u_KSmall}. Therefore $A_2\le 3$. This shows that $\scrG_\scrD(M)\le 1+2+2+3=8$.
	
	\item Assume $S_2=0$. By Proposition \ref{prop.LargeTSmallAngle}, we have that $A_2\le 5$, and so $\scrG_\scrD(M)\le 1+2+0+5=8$.
	\end{enumerate}
	
	\vspace*{5bp}
	\item Assume  $S_1=0$. Then $A_1\le 4$ by Propositions \ref{prop.SmallTSmallAngle} and \ref{prop.AsymCone1}. But $A_1=4$ contradicts Proposition \ref{prop.v_KSmallT}, so in fact $A_1\le 3$. If $A_1=3$ then by Proposition \ref{prop.v_KSmallT}  we have that \eqref{eqn.u_KSmall} holds and, for the problem of bounding $A_2+S_2$, we are in the same position as we just were in Case A. By exactly the same arguments, we therefore have that $A_2+S_2\le 6$, and that $\scrG_\scrD(M)\le 0+3+6=9$. \vspace*{5bp}
	
	\noindent If $A_1\le 2$ then $S_2\le 5$ by Proposition \ref{prop.SymConeUB}, and we again break the problem into cases.
	\begin{enumerate}
	\item If $S_2=5$ then $\vec{v}_K\in\RR_{>0}\mc{S}$, and we claim that $A_2\le 2$. To see why, suppose that $1\le u_1<\cdots <u_4<K$ are chosen so that $\vec{v}_{i_1},\ldots,\vec{v}_{i_4}$ are in direction $\mc{S}$. Suppose that at least three of these vectors all have their $u_i$ values either in $(-1,-1/2]$ or in $[1/2,1)$ and without loss of generality, suppose these values lie in $(-1,-1/2]$. Then as before, since at least two of these vectors both lie either above or below the $x$-axis, there cannot be a value of $1\le j<K$ with $u_j\in (-1,-1/2]$ and with $\vec{v}_j\notin\RR_{>0}\mc{S}.$ Furthermore, by Proposition \ref{prop.LargeTSmallAngle}, there can be at most two values of $1\le j< K$ with $u_j\in [1/2,1)$ and $\vec{v}_j\notin\mc{S}$, which gives $A_2\le 2$. The other possibility is that two of values of $u_{i_1},\ldots ,u_{i_4}$ lie in $(-1,-1/2]$, and two lie in $[1/2,1)$. In this case, by the same arguments as above, there cannot be two values of $1\le j< K$ with $\vec{v}_j\notin\RR_{>0}\mc{S}$ and with $u_j\in (-1,-1/2]$, and neither can there be two values with $u_j\in [1/2,1)$. This again gives $A_2\le 2$. This shows that if $A_1\le 2$ and $S_2=5$ then $\scrG_\scrD(M)\le 0+2+5+2=9$. \vspace*{5bp}
	
	\item If $S_2=3$ or $4$ then we claim that $A_2\le 3$. To see why this is true, suppose by way of contradiction that $S_2=3$ or $4$ and that $\vec{v}_{i_1},\ldots,\vec{v}_{i_4}\notin\RR_{>0}\mc{S}$ are distinct vectors with corresponding $u_i$ values all in $(-1,-1/2]\cup [1/2,1)$. We cannot have four of these values all in $(-1,-1/2]$ or all in $[1/2,1)$. If three of them all lie in one of these intervals, then without loss of generality let us suppose that the interval is $(-1,-1/2]$, the corresponding indices are $i_2,i_3,$ and $i_4$ and that the largest of these indices is $i_4$. We must then have that $i_4=K$. We have that $u_{i_1}$ lies in $[1/2,1)$, so to avoid contradiction there can be at most two values of $1\le \ell<K$ with $u_\ell\in[1/2,1)$ and $\vec{v}_\ell\in\RR_{>0}\mc{S}$. Then, there must be at least one value of $\ell$ with $u_\ell\in (-1,-1/2]$ and $\vec{v}_\ell\in\RR_{>0}\mc{S}$. However, we then conclude by previous arguments that either this $u_\ell$ value, or one of $u_{i_1}$ or $u_{i_2},$ must equal $u_K$. This is a contradiction. We are left with the possibility that two of the numbers $u_{i_1},\ldots ,u_{i_4}$ lie in $(-1,-1/2]$, and that the other two lie in $[1/2,1)$. This then implies that there is at most one value of $1\le\ell\le K$ with $\vec{v}_\ell\in\RR_{>0}\mc{S}$ and $u_\ell\in (-1,-1/2]$, and similarly at most one with $\vec{v}_\ell\in\RR_{>0}\mc{S}$ and $u_\ell\in [1/2,1)$. This contradicts the assumption that $S_2=3$ or $4$, so we conclude that in this case $A_2\le 3$. This gives that $\scrG_\scrD(M)\le 0+2+4+3=9$. \vspace*{5bp}
	
	\item Finally, if $S_2\le 2$ we use the bound $A_2\le 5$ and obtain $\scrG_\scrD(M)\le 0+2+2+5=9.$
	\end{enumerate}
\end{enumerate}\vspace*{10bp}

\noindent \textbf{Case (C2)}, $\tau=4\pi/3$: In this case $\psi=2\pi/3$ and $\phi=\pi/3$, and so $\lceil\frac{\phi}{\pi/3}\rceil= 1$. There are two asymmetric cones, $\mc{A}_{-1}=[-\pi/3,0)$ and $\mc{A}_{1}=[0,\pi/3).$\vspace*{5bp}

\begin{enumerate}[A.]
	\item Assume $S_1=1$. Then \eqref{eqn.u_KSmall} holds, and as in the previous Case A we have $A_1\le 2$ and $S_2\le 2$. If $S_2=2$ then choose $1\le i,j\le K-1$ with $i\not=j$ and $\vec v_i,\vec v_j\in\RR_{>0}\mc{S}$. If $u_i$ and $u_j$ both lie in $(-1,-1/2]$ then there can be at most one value of $1\le k\le K-1$ with $\vec v_k\not\in\RR_{>0}\mc{S}$ and with $u_k\in(-1,-1/2]$, and at most two values with $u_k\in [1/2,1)$, which gives $A_2\le 3$. Similarly if $u_i$ and $u_j$ both lie in $[1/2,1)$. If $u_i\in(-1,-1/2]$ and $u_j\in[1/2,1)$, or vice-versa, then there can be at most one value of $1\le k\le K-1$ with $\vec v_k\not\in\RR_{>0}\mc{S}$ and with $u_k\in(-1,-1/2]$, and also at most one value with $u_k\in [1/2,1)$, so in this case $A_2\le 2$. This gives the bound $\scrG_\scrD(M)\le 1+2+2+3=8.$ \vspace*{5bp}
	
	\noindent If $S_2=0$ or $1$ we use the bound $A_2\le 4$, which follows from Proposition \ref{prop.LargeTSmallAngle} and \eqref{eqn.u_KSmall}. Thus $\scrG_\scrD(M)\le 1+2+1+4=8$. \vspace*{5bp}
	
	\item Assume $S_1=0$. Then $A_1\le 3$. If $A_1=3$ then \eqref{eqn.u_KSmall} holds and, for the problem of bounding $A_2+S_2$, we may follow exactly the same arguments just used in Case A to obtain the bound $A_2+S_2\le 5$. Therefore $\scrG_\scrD(M)\le 0+3+5 =8$.\vspace*{5bp}
	
	\noindent If $A_1\le 2$ we use the bound $S_2\le 3$ from Proposition \ref{prop.SymConeUB}. By arguments similar to those in the previous Case B (b) and (c): If $S_2=3$ then $A_2\le 3$ and $\scrG_\scrD(M)\le 0+2+3+3=8$. If $S_2=1$ or $2$ then $A_2\le 4$ and $\scrG_\scrD(M)\le 0+2+2+4=8$. If  $S_2=0$ then $A_2\le 5$ and $\scrG_\scrD(M)\le 0+2+0+5=7$.
\end{enumerate}\vspace*{10bp}

\noindent \textbf{Case (C3)}, $\pi<\tau<4\pi/3$: In this case $2\pi/3<\psi<\pi$ and $0<\phi<\pi/3$. Now there are three asymmetric subsets, $\mc{A}_{-1}=[-\psi/2,-\pi/6),~\mc{A}_{0}=[-\pi/6,\pi/6),$  and  $\mc{A}_{1}=[\pi/6,\psi/2)$. The argument is similar to those given previously, except that to bound the number of vectors $\vec{v}_i$ in the direction of $\mc{A}_{-1}\cup\mc{A}_1$ with  $u_i\in (-1/2,1/2)$, we must now use Proposition \ref{prop.AsymCone2}. Also, we do not try to optimize the argument as much, since the bound in Proposition \ref{prop.AsymCone2} is probably already sub-optimal.\vspace*{5bp}

\begin{enumerate}[A.]
	\item If $S_1=1$ then \eqref{eqn.u_KSmall} holds. As before, Proposition \ref{prop.v_KSmallT} guarantees that all vectors $\vec{v}_i$ with $i<K$ have corresponding $u_i$ values all in $(-1/2,0]$ or all in $[0,1/2)$. Therefore, by Propositions \ref{prop.SmallTSmallAngle} and \ref{prop.AsymCone1} there is at most one such vector pointing in the direction of $\mc{A}_0$. Similarly, by Propositions \ref{prop.SmallTSmallAngle} and \ref{prop.AsymCone2} (applied to each of the cones $\mc{A}_{-1}$ and $\mc{A}_1$), the number of such vectors lying in the direction of $\mc{A}_{-1}\cup\mc{A}_1$ is at most
	\begin{equation}
		2+2\left\lfloor\frac{\sin (\tau/2+\pi/6)}{\sin(\tau-\pi)}\right\rfloor.
	\end{equation}
	This gives the bound
	\begin{equation}
	A_1\le 3+2\left\lfloor\frac{\sin (\tau/2+\pi/6)}{\sin(\tau-\pi)}\right\rfloor.
	\end{equation}
	Since $\lceil\frac{\phi}{\pi/3}\rceil= 1$, Proposition \ref{prop.SymConeUB} implies that $S_2\le 2$. Also, Proposition \ref{prop.LargeTSmallAngle} (applied to each of the three asymmetric sets) implies that $A_2\le 6$, so we have that
	\begin{equation}\label{eqn.C3g_NBnd}
	\scrG_\scrD(M)\le 12+2\left\lfloor\frac{\sin (\tau/2+\pi/6)}{\sin(\tau-\pi)}\right\rfloor.
	\end{equation}
	
	\item Finally, suppose that $S_1=0$. In this case, again by Propositions \ref{prop.SmallTSmallAngle}-\ref{prop.AsymCone2}, there can be at most
\begin{equation}
		3+2\left\lfloor\frac{\sin (\tau/2+\pi/6)}{\sin(\tau-\pi)}\right\rfloor
\end{equation}
indices $1\le i\le K$ with $\vec v_i$ in an asymmetric set and with $u_i\in [0,1/2)$. In addition to these vectors, if there is any other vector $\vec v_j$ in an asymmetric set, with $u_j\in (-1/2,0)$, then it would follow from Proposition \ref{prop.v_KSmallT} that $j=K$. Therefore there can be at most one such vector $\vec v_j$. The same argument applies with the intervals $[0,1/2)$ and $(-1/2,0)$ replaced by $(-1/2,0]$ and $(0,1/2)$, respectively, and this gives the bound
\begin{equation}\label{eqn.A_1IneqCaseC3Final}
A_1\le 4+2\left\lfloor\frac{\sin (\tau/2+\pi/6)}{\sin(\tau-\pi)}\right\rfloor.
\end{equation}
If equality holds in this inequality then we know that \eqref{eqn.u_KSmall} holds, we are in the same situation just encountered in Case A, and we again use the bounds $S_2\le 2$ and $A_2\le 6$, arriving at the same bound \eqref{eqn.C3g_NBnd} for $\scrG_\scrD(M)$.

If there is strict inequality in \eqref{eqn.A_1IneqCaseC3Final} then Proposition \ref{prop.SymConeUB} implies that $S_2\le 3,$ and it also tells us that if $S_2=3$ then $\vec v_K$ lies in a symmetric set. Similarly, we have from Proposition \ref{prop.LargeTSmallAngle} that $A_2\le 7$, and that if $A_2=7$ then $\vec v_K$ lies in an asymmetric set. Therefore we cannot have both $S_2=3$ and $A_2=7$. This gives that $S_2+A_2\le 9$, which again implies that \eqref{eqn.C3g_NBnd} holds.

\end{enumerate}

\section{Explicit upper bounds in dimension $d=2$, part 2}\label{sec.2DBddPfPart2}

Throughout this section we set $d=2$ and $\scrD=\SS_1^1$. In this case the upper bound of 6 obtained in the previous section falls just short of establishing the claimed five distance theorem. We will first deduce a little more information about the possible distances which can occur in this case. The goal of this section is to prove the following theorem, which will thereby complete the proof of Theorem \ref{thm.2DBoundedGaps} and the $d=2$ case of Theorem \ref{thm.HighDBoundedGaps}.

\begin{thm}\label{thm.2DBoundedGapsFULL2}
For any $M\in \mathrm{SL}_3(\RR)$ we have that
\begin{equation}
\scrG_{\SS_1^1}(M)\le 5.
\end{equation}
\end{thm}

Since we are dealing with the case $\scrD=\SS_1^1$ we can assume, by replacing each vector $(u_i,\vec{v}_i)$ by its negative if necessary, that $u_i\ge 0$ for $1\le i\le K$. For simplicity we make this assumption for the duration of this section. It is convenient at this point to gather together some additional properties which must be satisfied by the vectors $\vec v_i$.
\begin{prop}\label{prop.vSumAndDiff}
If $\scrG_{\SS_1^1}(M)= K$ then, for all $1\le i<j\le K-1$,
\begin{align}
|\vec v_i-\vec v_j|&>|\vec v_j|,\label{eqn.vBd1}
\end{align}
and for all $1\le i<j<k\le K-1$,
\begin{align}
|\vec v_i-\vec v_j-\vec v_k|&\ge|\vec v_k|.\label{eqn.vBd2}
\end{align}
\end{prop}
\begin{proof}
Let $1\le i<j\le K-1$. To prove \eqref{eqn.vBd1} first note that by Proposition \ref{prop.SymSmallT}, together with the above mentioned convention that $u_i, u_j\ge 0$, we have
\begin{equation}
0<u_i-u_j<1/2.
\end{equation}
Since $\scrG_{\SS_1^1}(M)=K$ we must have that $u_j\ge 1/2$. Therefore, applying Proposition \ref{prop.ADDI2} to the vector $(u_i-u_j,\vec{v}_i-\vec{v}_j)\in\scrQ_\scrD(\SS_1^1)$ gives that
\begin{equation}
|\vec v_i-\vec v_j|\ge|\vec v_j|.
\end{equation}
If there were equality in this inequality then, by the argument used in the proof of Proposition \ref{prop.SymSmallT}, we would have that $j=K$, which is a contradiction. Therefore, the strict inequality \eqref{eqn.vBd1} holds.

Next, to prove \eqref{eqn.vBd2}, let $1\le i<j<k\le K-1$. We have that $1/2\le u_k<u_j<u_i<1$, and therefore that
\begin{equation}
0<u_k+u_j-u_i<u_k.
\end{equation}
Applying Proposition \ref{prop.ADDI2} to the vector $(u_k+u_j-u_i,\vec{v}_k+\vec{v}_j-\vec v_i)\in\scrQ_\scrD(\SS_1^1)$ gives that
\begin{equation}
|\vec{v}_k+\vec{v}_j-\vec v_i|\ge|\vec v_k|,
\end{equation}
which proves the result.
\end{proof}
As a corollary of Proposition \ref{prop.vSumAndDiff}, we also deduce the following result.
\begin{prop}\label{prop.vAngles}
If $\scrG_{\SS_1^1}(M)= K$ then, for all $1\le i<j\le K-1$, the angle between $\vec{v}_i$ and $\vec{v}_j$ must be greater than $\pi/3$. Also, for all $1\le i<j<k\le K-1$, the vector $\vec v_i$ does not lie in the positive cone determined by $\vec v_j$ and $\vec v_k$.
\end{prop}
\begin{proof}
For the first part of the proposition, if $1\le i<j\le K-1$, then the fact that the angle between $\vec v_i$ and $\vec v_j$ must be greater than $\pi/3$ follows \eqref{eqn.vBd1} together with Proposition \ref{prop.LengthReduc}.

\begin{figure}[ht]
	\centering
	\def\svgwidth{0.7\columnwidth}
	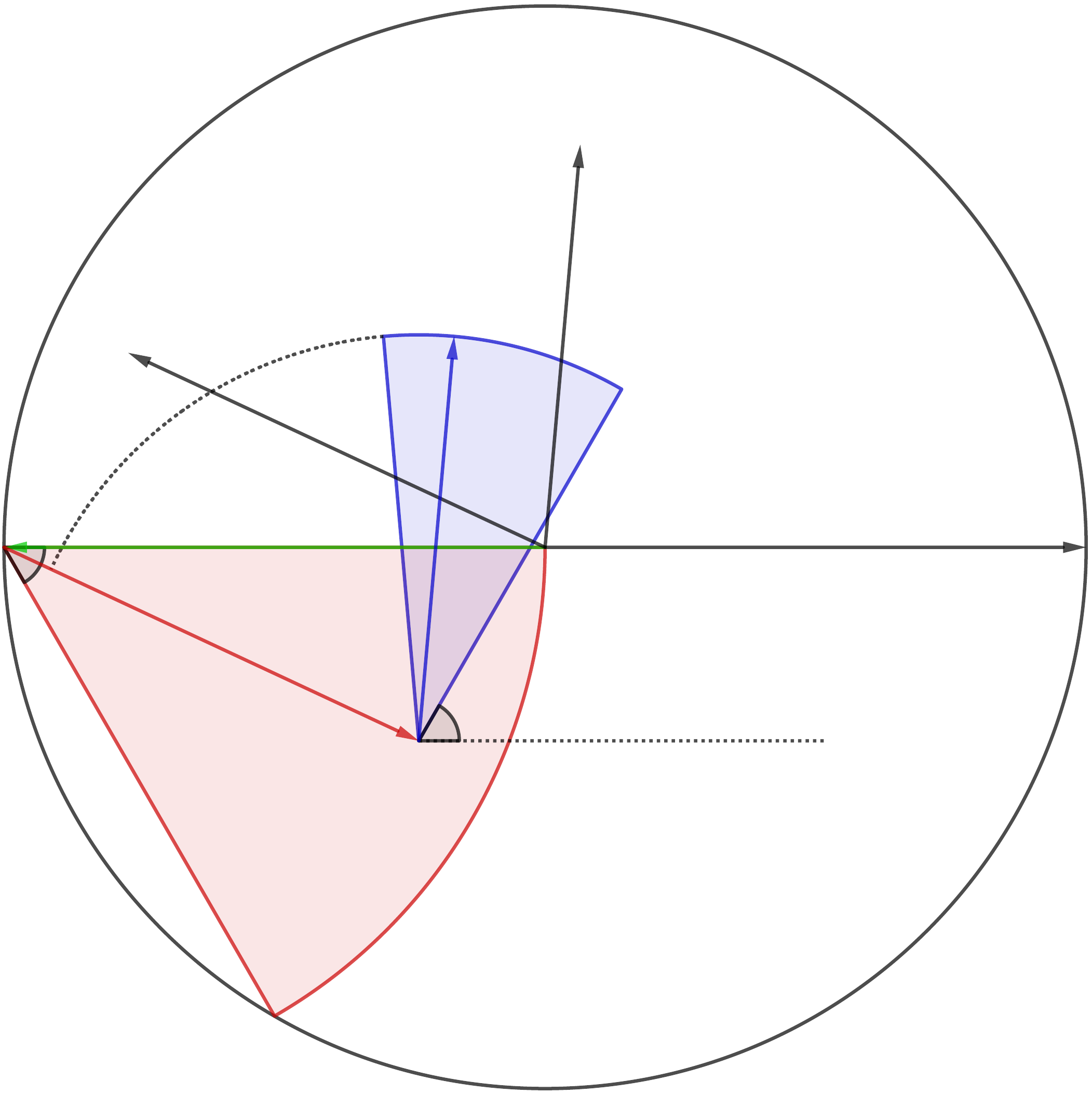
	\caption{Illustration corresponding to the contradictory hypothesis that $\vec v_i$ lies in the positive cone determined by $\vec v_j$ and $\vec v_k$.}\label{fig.GapsProp1}
\end{figure}

For the second part, suppose by way of contradiction that $1\le i<j<k\le K-1$, and that the vector $\vec v_i$ does lie in the positive cone determined by $\vec v_j$ and $\vec v_k$. Note that $|\vec v_i|<|\vec v_j|<|\vec v_k|$ and that, by the first part of the proposition, the angle between $\vec v_j$ and $\vec v_k$ is greater than $2\pi/3$. With these observations in mind, consider Figure \ref{fig.GapsProp1}. The figure is rotated so that the vector $\vec v_k$ is aligned with the positive $x$-axis. Depending on the orientation of the vectors involved, it may also be reflected about the $x$-axis. The vector $-\vec v_j-\vec v_k$ must lie in the sector indicated in red. Once $\vec{v}_j$ is chosen, the vector $\vec v_i-\vec v_j-\vec v_k$ must lie in a sector of the circle of radius $|\vec v_i|$ centered at $-\vec v_j-\vec v_k$, as indicated by the blue region in the figure. However, no matter what choice is made for $\vec v_j$, this sector will lie completely within the open disc of radius $|\vec v_k|$ centered at the origin. 

Since this contradicts \eqref{eqn.vBd2}, we conclude that $\vec v_i$ can not lie in the positive cone determined by $\vec v_j$ and $\vec v_k$.
\end{proof}

\begin{proof}[Proof of Theorem \ref{thm.2DBoundedGapsFULL2}.]
Suppose, contrary to the statement of the theorem, that $\scrG_{\SS_1^1}(M)\ge 6$. Consider the collection of vectors $\vec v_i$, with $1\le i\le 5$. We will say that two vectors from this collection are \textit{consecutive} if there is no other vector from the collection which lies in their positive cone. Every vector in the collection is consecutive to two others. By Proposition \ref{prop.vAngles}, the angle between any pair of consecutive vectors is greater than $\pi/3$. Since there are five vectors in the collection, it follows that if $i,j,$ and $k$ are distinct indices with $1\le i,j,k\le 5$ and if $\vec v_i$ is consecutive to both $\vec v_j$ and $\vec v_k$, then the angle between $\vec v_j$ and $\vec v_k$ is less than $\pi$. In other words, in the situation just described, the vector $\vec v_i$ lies in the positive cone determined by $\vec v_j$ and $\vec v_k$. Therefore, by Proposition \ref{prop.vAngles}, if $i,j,$ and $k$ are distinct and if $\vec v_i$ is consecutive to both $\vec v_j$ and $\vec v_k$, then it must be the case that $i>\min\{j,k\}.$ However, the vector $\vec v_1$ is consecutive to two vectors $\vec v_j$ and $\vec v_k$, with $1<j<k$, and this gives a contradiction. Therefore we conclude that $\scrG_{\SS_1^1}(M)\le 5$.
\end{proof}

\section{Explicit upper bounds in dimension $d>2$}\label{sec.HighDBddPf}
Let $G=\SL(d+1,\RR)$ and $\Gamma=\SL(d+1,\ZZ)$. As in the proof of Theorem \ref{thm.2DBoundedGapsFULL}, for each $M\in\GamG$ we suppose that $K\in\N$ and $\{(u_i,\vec{v}_i)\}_{i=1}^K\subseteq M$ are chosen so that conditions (V1)-(V3) hold. In this section we will prove the following statement, which by \eqref{keyineq} implies Theorem \ref{thm.HighDBoundedGaps} in dimension $d\geq 3$.

\begin{thm}\label{thm.HighDBoundedGapsGEN}
Let $d\ge 3$ and $\scrD=\SS_1^{d-1}$. Then for any $M\in G$ we have that
\begin{equation}
\scrG_\scrD(M)\le \sigma_d+1.
\end{equation}
\end{thm}

We will use the following analogues of Propositions \ref{prop.SymSmallT} and \ref{prop.LargeTSmallAngle}.

\begin{prop}\label{prop.HighDSmallT}
Let $d\ge 3$ and $\scrD=\SS_1^{d-1}$. If $(u,\vec v)\in \scrQ_\scrD(M)$ with $u\in(-1/2,1/2)$, then $|\vec{v}_K|\le |\vec v|$. It follows that if, for some integer $1\le i\le K$ we have that $u_i\in(-1/2,1/2)$, then we must have that $i=K$.
\end{prop}
\begin{proof}
Suppose first that $u\in(-1/2,0]$. Then $-u\in[0,1/2)$ and $(-u,-\vec v)\in \scrQ_\scrD(M)$. Therefore, for any $t\in (0,1)$, any shortest vector in $\scrQ_\scrD(M,t)$ can have length at most $|\vec v|$. If $u\in [0,1/2)$ the argument is symmetric, so this verifies the first claim of the proposition.

For the second claim, apply the first with $(u,\vec v)=(u_i,\vec{v}_i)$. Then it follows from properties (V1)-(V3) that $i=K$.
\end{proof}

\begin{prop}\label{prop.HighDSameCone}
Let $d\ge 3$ and $\scrD=\SS_1^{d-1}$. If $1\le i,j\le K$ and
\begin{equation}-1<u_i<u_j\leq -1/2 \qquad\text{or}\qquad 1/2\leq u_j<u_i<1,\end{equation}
then the angle between $\vec{v}_i$ and $\vec{v}_j$ is greater than $\pi/3$.
\end{prop}
\begin{proof}
Suppose the hypotheses of the proposition are satisfied and that $u_i,u_j\in (-1,-1/2]$ with $u_i<u_j$. Then conditions (V1) and (V3) imply that $i<j$ and that $|\vec{v}_i|<|\vec{v}_j|$ (this follows, for example, from the argument used in the proof of Proposition \ref{prop.ADDI2}). If the angle between $\vec{v}_i$ and $\vec{v}_j$ were less than or equal to $\pi/3$ then by Proposition \ref{prop.LengthReduc} we would have that
\begin{equation}
|\vec{v}_i-\vec{v}_j|< |\vec{v}_j|.
\end{equation}
Since $(u_i-u_j,\vec{v}_i-\vec{v}_j)\in \scrQ_\scrD(M)$ satisfies $u_i-u_j\in(-1/2,1/2)$, we could then deduce from Proposition \ref{prop.HighDSmallT} that
\begin{equation}
|\vec{v}_K|\le |\vec{v}_i-\vec{v}_j|<|\vec{v}_j|,
\end{equation}
which contradicts (V1). Therefore the angle between $\vec{v}_i$ and $\vec{v}_j$ is greater than $\pi/3$. The argument is symmetric if $u_i,u_j\in [1/2,1)$ with $u_i>u_j$.
\end{proof}

\begin{proof}[Proof of Theorem \ref{thm.HighDBoundedGapsGEN}]
By Proposition \ref{prop.HighDSmallT}, there is at most one value of $1\le i\le K$ with $u_i\in (-1/2,1/2)$. If such a value exists then it must be $K$.

For each $1\le i\le K-1$ let
\begin{equation}
	\vec{x}_i=\frac{\vec{v}_i}{|\vec{v}_i|}\in\SS_1^{d-1}.
\end{equation}
By Proposition \ref{prop.HighDSameCone} any points $\vec{x}_i$ and $\vec{x}_j$ with $i\not=j$ are separated by an angle greater than $\pi/3$. Therefore the collection of spheres of radius $1/2$ centered at the points $\vec{x}_i$, for $1\le i\le K-1$, do not overlap, and they are all tangent to the sphere of radius $1/2$ centered at the origin. It follows that
\begin{equation}
	K-1\le\sigma_d,
\end{equation}
and this completes the proof of Theorem \ref{thm.HighDBoundedGapsGEN}.
\end{proof}

\section{Continuity and local upper/lower bounds}\label{uplow}

We now turn to the case of general $\scrD\subseteq\SS_1^{d-1}$. We will show in this section that there are choices of $\scrD$ for which $F_\scrD(M,t)$ is unbounded. First we establish local upper bounds (i.e. upper bounds for $M$ on compacta) that hold for general $\scrD$.

\begin{prop}\label{prop:cont}
Suppose that $\scrD\subseteq\SS_1^{d-1}$ has non-empty interior and that $\scrC\subset\GamG\times(0,1)$ is compact. Then the following must hold:
\begin{enumerate}[{\rm (i)}]
	\item There exists a number $\kappa(\scrC)>0$ such that $F_\scrD(M,t) < \kappa(\scrC)$ if $(\Gamma M,t)\in \scrC$, and
	\item $F$ is continuous at every point $(\Gamma M,t) \in \scrC$ with
	\begin{equation}\label{exset}
	(\ZZ^{d+1} M\setminus\{0\})
	\cap \partial((-t,1-t) \times (0,\kappa(\scrC)]\scrD) = \emptyset .
	\end{equation}
\end{enumerate} 
\end{prop}

We emphasise that in (ii), relation \eqref{exset} needs to be verified only for one specific but arbitrary $\kappa(\scrC)>F_\scrD(M,t)$.

\begin{proof} It follows from the proof of Proposition \ref{WellProp} above that $F(M,t)$ is finite, for any choice of $M$ and $t$. The actual bound obtained by that argument depends on the choice of $\epsilon$ in the proof. By Mahler's compactness criterion, since $\scrC$ is compact, it is possible to choose a single value of $\epsilon>0$ which works for every $(M,t)\in\scrC$. This implies the existence of a constant $\kappa(\scrC)$ satisfying the condition in (i).
	
The proof of (ii) is then analogous to the proof of \cite[Prop.~2]{HaynMar2020}.
\end{proof}

Let us now extend the above uniform upper bound to all $t\in(0,1)$, with $M$ in a compact set.
Given a bounded subset $\scrA\subset\RR^{d+1}$ with non-empty interior, and $M\in G$, we define the {\em covering radius} (also called {\em inhomogeneous minimum})
\begin{equation}
\rho(M,\scrA)=\inf\{ \theta>0 \mid \theta\scrA+\ZZ^{d+1} M = \RR^{d+1} \}.
\end{equation}
Because $\scrA$ has non-empty interior, $\rho(M,\scrA)<\infty$. We will in the following take $\scrA$ to be of the form $\scrA=(0,1)\times (0,r]\scrD$. Then $\scrA\subset \lambda \scrA$ for any $\lambda> 1$, and thus $\theta\scrA+\ZZ^{d+1} M = \RR^{d+1}$ for every $\theta>\rho(M,\scrA)$. Therefore, for such $\theta,$ the set $\theta\scrA + \vec x$ intersects $\ZZ^{d+1} M$ in at least one point, for every $\vec x\in\RR^{d+1}$. 

For a given set $\scrC\subset\GamG$, we define
\begin{equation}
\overline\rho(\scrC,\scrA) = \sup_{\Gamma M\in\scrC} \rho(M,\scrA).
\end{equation}
It is well known that $\overline\rho(\scrC,\scrA)<\infty$ for every compact $\scrC\subset\GamG$. This follows, for example, from the comments about finiteness and continuity of inhomogeneous minima of balls at the tops of pages 231 and 234 of \cite{GrubLekk1987} (an upper bound for $\overline\rho(\scrC,\scrA)$ is obtained from an upper bound with $\scrA$ replaced by a ball contained in $\scrA$). For $\theta>0$, set
\begin{equation}\label{D-def}
D(\theta) =
\begin{pmatrix}
\theta^d & 0 \\ 0 & \theta^{-1} \bm{1}_d
\end{pmatrix}
\in G.
\end{equation}

\begin{prop}\label{prop:twoB}
Let $\scrD\subseteq\SS_1^{d-1}$ with non-empty interior. Assume $\scrC\subset\GamG$ is compact, and $\theta>\overline\rho(\scrC,(0,1)\times(0,1]\scrD)$. Then
\begin{equation}\label{eqn.FUppBd}
F_\scrD(M,t) \leq \theta^{d+1}
\end{equation}
for $\Gamma M\in \scrC D(\theta)^{-1}$ and $t\in(0,1)$.
\end{prop}

\begin{proof}
For each $t\in (0,1)$, the set
\begin{equation}
	\scrA_{t,\theta}=(-t,1-t)\times (0,\theta^{d+1}] \scrD\subset\RR^{d+1}
\end{equation} has non-empty interior. For any $\Gamma M\in\scrC D(\theta)^{-1}$ and $t\in(0,1)$, it follows from the definition of the function $F_\scrD$ that, if $\scrA_{t,\theta}$ intersects $\ZZ^{d+1} M$ in at least one point, then \eqref{eqn.FUppBd} holds. Now $\scrA_{t,\theta}\cap\ZZ^{d+1} M\neq \emptyset$ is equivalent to $\theta^d \scrA_{t,1}\cap\ZZ^{d+1} M D(\theta) \neq \emptyset$. The latter holds because the assumption that $\theta>\overline\rho(\scrC,(0,1)\times(0,1]\scrD)$ implies that $\theta\scrA_{t,1}\cap\ZZ^{d+1} M' \neq \emptyset$ for every $\Gamma M'\in\scrC$. Now choose $M'=M D(\theta)$ and note that $\Gamma M D(\theta) \in\scrC$ by assumption.
\end{proof}

As a consequence of the previous proposition, we have the following result.

\begin{prop}\label{prop:three}
Let $\scrD\subseteq\SS_1^{d-1}$ have non-empty interior. Assume $\scrC\subset\GamG$ is compact, and $\theta>\overline\rho(\scrC,(0,1)\times(0,1]\scrD)$. Then there is a constant $C_{\theta,\scrD}<\infty$ such that
\begin{equation}
\scrG_\scrD(M) \leq C_{\theta,\scrD}
\end{equation}
for $\Gamma M\in\scrC D(\theta)^{-1}$.
\end{prop}

\begin{proof}
It follows from Proposition \ref{prop:twoB}, together with the definition of the function $\scrG_\scrD$, that for any $\Gamma M\in\scrC D(\theta)^{-1}$, the quantity $\scrG_\scrD(M)$ is bounded above by the number of lattice points in $\ZZ^{d+1} M$ which lie in the region
\begin{equation}
	(-1,1) \times [0,\theta^{d+1}]\scrD\subseteq\R^{d+1}.
\end{equation}
In view of Mahler's criterion, the number of such lattice points is uniformly bounded above for all $\Gamma M$ in the compact subset $\scrC D(\theta)^{-1}$ of $\GamG$.
\end{proof}

The key point is now that, in addition to the above upper bounds, we can find open sets $\scrU\subset\GamG$, on which $\scrG_{\scrD,N}(M)$ can exceed any given value. This requires however that $\scrD$ is contained in a hemisphere.
Let $\scrH\subset\SS_1^{d-1}$ be an (arbitrary) open hemisphere, and $\scrD\subset\SS_1^{d-1}$ with non-empty interior so that $\scrD^\cl\subset\scrH$.
Choose $d+1$ row vectors $\vec e_0,\vec e_1,\ldots,\vec e_d\in\SS_1^{d-1}$ with the properties 
\begin{enumerate}[(i)]
\item $\vec e_1\in\scrD^\circ$ and $\vec e_0\in\scrH\setminus\scrD$ such that $\vec e_0\cdot \vec e_1>0$;
\item $\vec e_2,\ldots,\vec e_d\notin\scrD$ such that $(\vec e_1,\ldots,\vec e_d)$ forms an orthonormal basis of $\RR^d$ with
$$\det\begin{pmatrix} \vec e_1\\ \vdots \\ \vec e_d\end{pmatrix}=1.$$
\end{enumerate}
Given $\epsilon>0$, define the matrix
\begin{equation}
M_\epsilon = \begin{pmatrix} \epsilon & -\epsilon \vec e_0 \\ 0 & \vec e_1 \\ 0 & \epsilon^{-1/(d-1)}\vec e_2 \\ \vdots & \vdots \\ 0 & \epsilon^{-1/(d-1)}\vec e_d\end{pmatrix} .
\end{equation}
Note that $\det M_\epsilon = 1$, and thus the vectors
$$
\vec b_0=(\epsilon,-\epsilon \vec e_0),\; \vec b_1=(0,\vec e_1),\; \vec b_2=(0,\epsilon^{-1/(d-1)}\vec e_2),\; \ldots,\; \vec b_d=(0,\epsilon^{-1/(d-1)}\vec e_d)
$$
form a basis of the unimodular lattice $\scrL_\epsilon=\ZZ^{d+1} M_\epsilon$.

Let $L(x)$ denote the largest integer strictly less than $x$. That is, in terms of the floor function $L(x)=\lfloor x\rfloor$ if $x\notin\ZZ$ and $L(x)=x-1$ if $x\in\ZZ$.

\begin{prop}\label{prop:unbddNEW}
Let $\scrD\subset\scrH$ and $M_\epsilon$ be as above. Then there exist $\lambda\in(0,1)$ and $\epsilon_0>0$ such that, for any $\epsilon\in(0,\epsilon_0]$ and $t\in(\lambda, 1)$, 
\begin{enumerate}[{\rm (i)}]
\item $F_\scrD(M_\epsilon,t) = |\vec e_1-\epsilon L(\epsilon^{-1}(1-t)) \vec e_0|$ and
\item $F_\scrD$ is continuous at $(\Gamma M_\epsilon,t)\in\GamG\times (0,1)$ if $t\notin 1+\epsilon \ZZ$.
\end{enumerate}
\end{prop}

\begin{proof} Fix $s_0\in (0,\vec e_0\cdot \vec e_1)$, and denote by $s_-\leq 0\leq s_+$ the infimum and supremum of all $s$ such that $\vec e_1-s \vec e_0\in\RR_{>0}\scrD$; since $e_1\in\scrD^\circ$ we have $s_-< 0 < s_+$. Note also that
	\begin{equation}\label{elan}
		\frac{d}{ds} |\vec e_1-s \vec e_0| = \frac{s -  \vec e_0\cdot \vec e_1}{|\vec e_1-s \vec e_0|}< 0
	\end{equation}
	if $s\leq s_0$.
	
Proof of (i): We are interested in the lattice points from $\ZZ^{d+1} M_\epsilon$ contributing to \eqref{QMt}, i.e.,
\begin{equation} \label{QMteps}
\scrQ_\scrD(M_\epsilon,t) = \big\{ (u,\vec v)\in\ZZ^{d+1} M_\epsilon\;\big| \; -t< u < 1-t, \;\vec v \in \RR_{>0}\scrD \big\} ,
\end{equation}
and in particular those with minimal $|\vec v|$. We begin with those elements of the form 
\begin{equation}\label{tform}
(u,\vec v) = m_0 \vec b_0 + m_1\vec b_1=(\epsilon m_0 ,m_1 \vec e_1-\epsilon m_0 \vec e_0),\qquad m_0\in\ZZ,\; m_1\in\ZZ_{\geq 1}.
\end{equation}
By construction, we have $m_1 \vec e_1-\epsilon m_0 \vec e_0\in\RR_{>0}\scrD^\circ$ if and only if $s_-<\epsilon m_0/m_1 < s_+$; and $m_1 \vec e_1-\epsilon m_0 \vec e_0\in\RR_{>0}\scrD^\cl$ if and only if $s_-\le\epsilon m_0/m_1 \leq s_+$.
In view of \eqref{elan}, the length $|\vec e_1-\epsilon m_0/m_1 \vec e_0|$ is strictly decreasing (as a function of $m_0/m_1$) if $\epsilon m_0/m_1\leq s_0$. Let us restrict our attention to those $t$ for which $1-t< \min\{s_0,s_+\}$. Then the condition $\epsilon m_0< 1-t$ implies $\epsilon m_0/m_1 < s_+$ and $\epsilon m_0/m_1\leq s_0$ for all $m_1\geq 1$. 
Hence the smallest value of $|m_1 \vec e_1-\epsilon m_0 \vec e_0|$ is obtained for $m_0=L(\epsilon^{-1}(1-t))$ and $m_1=1$.
In summary, we have shown thus far that for $t> 1- \min\{s_0,s_+\}$,
\begin{equation}
F_\scrD(M_\epsilon,t)\leq |\vec e_1-\epsilon L(\epsilon^{-1}(1-t)) \vec e_0| .
\end{equation}
Note that 
\begin{equation}
|\vec e_1-s \vec e_0| = \sqrt{ 1 - 2s\, \vec e_0\cdot \vec e_1 +s^2} < 1
\end{equation}
for $0<s\leq s_0$ and $\vec e_0\cdot \vec e_1>s_0$, and hence $F_\scrD(M_\epsilon,t)<1$.

What we need to establish now is that all other elements in $\scrQ_\scrD(M_\epsilon,t)$ that are not of the form \eqref{tform} have larger $|\vec v|$.
Consider first the set of vectors
\begin{equation}\label{tform1}
(u,\vec v) = m_0 \vec b_0 + m_1\vec b_1=(\epsilon m_0 ,m_1\vec e_1-\epsilon m_0 \vec e_0),\qquad m_0\in\ZZ,\; m_1\in\ZZ_{\leq 0},
\end{equation}
Since $\vec e_0\in\scrH\setminus\scrD$ and $\vec e_1\in\scrD^\circ$ we have $m_1\vec e_1-\epsilon m_0 \vec e_0\notin\RR_{>0}\scrD$ for all $m_1\leq 0$, and hence the corresponding vectors are not in $\scrQ_\scrD(M_\epsilon,t)$. 

Next consider the remaining cases
\begin{equation}\label{tform2}
(u,\vec v) = m_0 \vec b_0 + \cdots + m_d\vec b_d,\qquad m_0,m_1\in\ZZ,\; (m_2,\ldots,m_d)\in\ZZ^{d-1}\setminus\{0\} .
\end{equation}
We need to understand whether any of these vectors can lie in $\scrQ_\scrD(M_\epsilon,t)$ and satisfy $|\vec v|< 1$.
We make the following observations:
\begin{itemize}
	\item[(a)] The domain $(-1,1)\times (0,1)\scrD$ is bounded.
	\item[(b)] The vector $m_0 \vec b_0$ has bounded length $|m_0 \vec b_0|=\sqrt2\,  |\epsilon m_0|<\sqrt 2$ (since $-1<\epsilon m_0<1$).
	\item[(c)] The vector $m_1 \vec b_1 + \cdots + m_d\vec b_d$ has length at least $\epsilon^{-1/(d-1)}$ since $(m_2,\ldots,m_d)\in\ZZ^{d-1}\setminus\{0\}$.
\end{itemize}
Therefore, as long as $\epsilon$ is sufficiently small, $(-1,1)\times (0,1)\scrD$ will not contain any vectors of the form \eqref{tform2}. Since we have now considered all vectors for our restricted values of $t$, this establishes claim (i) with $\lambda=1-\min\{s_0,s_+\}$.

Proof of (ii): We need to establish that, for $\lambda<t<1$, $t\notin 1+\epsilon \ZZ$, the function $F_\scrD$ is continuous at $(\Gamma M_\epsilon,t)$. 
By Proposition \ref{prop:cont} (with $\scrC=\{(\Gamma M_\epsilon,t)\}$), it is sufficient to check that
\begin{equation}\label{exset2}
(\ZZ^{d+1} M_\epsilon\setminus\{0\})
\cap \partial((-t,1-t) \times (0,\kappa]\scrD) = \emptyset ,
\end{equation}
for any fixed choice of $\kappa>F_\scrD(\Gamma M_\epsilon,t)$. We fix $\kappa$ so that $F_\scrD(\Gamma M_\epsilon,t)<\kappa<1$ and
\begin{equation}
\kappa\notin \{ |m_1\vec e_1-\epsilon m_0 \vec e_0| \mid (m_0,m_1)\in\ZZ_{\geq 1}^2 \}.
\end{equation}

Since vectors with $(m_2,\ldots,m_d)\in\ZZ^{d-1}\setminus\{0\}$ are outside the bounded domain $((-t,1-t) \times (0,\kappa]\scrD)^\cl$ for $\epsilon$ sufficiently small (by the argument in the proof of fact (i) above), what we are aiming to show is equivalent to
\begin{equation}\label{libel0}
\{ (\epsilon m_0 ,m_1\vec e_1-\epsilon m_0 \vec e_0) \mid (m_0,m_1)\in(\ZZ^2\setminus\{0\}) \cap \partial((-t,1-t) \times (0,\kappa]\scrD)\} = \emptyset .
\end{equation}
By the same argument as above, for $m_1\leq 0$ we have $m_1\vec e_1-\epsilon m_0 \vec e_0\notin \RR_{\geq 0}\scrD^\cl$ (unless $(m_0,m_1)=0$, which is excluded). So we can assume $m_1\geq 1$ from now on. Then, for $m_0\leq 0$, we have by the monotonicity \eqref{elan} that
$|m_1\vec e_1-\epsilon m_0 \vec e_0|\geq m_1 |\vec e_1|\geq |\vec e_1| = 1$.
Therefore $m_1\vec e_1-\epsilon m_0 \vec e_0\notin [0,\kappa]\scrD^\cl$.
What remains is to check that
\begin{equation}\label{libel01}
\{ (\epsilon m_0 ,m_1\vec e_1-\epsilon m_0 \vec e_0) \mid (m_0,m_1)\in\ZZ_{\geq 1}^2 \cap \partial((-t,1-t) \times (0,\kappa]\scrD)\} = \emptyset .
\end{equation}
The truth of relation \eqref{libel01} is equivalent to the truth of both
\begin{equation}\label{libel1}
\{ (\epsilon m_0 ,m_1\vec e_1-\epsilon m_0 \vec e_0) \mid (m_0,m_1)\in\ZZ_{\geq 1}^2 \cap (\{-t,1-t\} \times [0,\kappa]\scrD^\cl)\} = \emptyset 
\end{equation}
and
\begin{equation}\label{libel2}
\{ (\epsilon m_0 ,m_1\vec e_1-\epsilon m_0 \vec e_0) \mid (m_0,m_1)\in\ZZ_{\geq 1}^2 \cap ([-t,1-t] \times \partial((0,\kappa]\scrD)))\}= \emptyset .
\end{equation}
The first relation \eqref{libel1} is automatically satisfied since (a) by assumption $\epsilon m_0\neq 1-t$ for any integer $m_0$, and (b) $\epsilon m_0\neq -t$ because $t>0$ and $m_0>0$. 

As to the second relation \eqref{libel2}, the statement $m_1\vec e_1-\epsilon m_0 \vec e_0\in \partial((0,\kappa]\scrD)$  
implies $s_-= \epsilon m_0/m_1$ or  $s_+= \epsilon m_0/m_1$ or $|m_1\vec e_1-\epsilon m_0 \vec e_0|=\kappa$. The first option is not possible since $s_-<0$, and third is excluded by assumption. If the second option holds, then, since $\epsilon m_0\in[-t,1-t]$, we have $m_1 s_+\leq 1-t$ and thus $s_+\leq 1-t$. But this contradicts our assumption $1-t<s_+$. Hence \eqref{libel2} holds and the proof is complete.
\end{proof}

The following lower bound on the number of distinct values of $t\mapsto F_\scrD(M,t)$ is a corollary of the previous proposition.

\begin{prop}\label{lowgapprop}
Let $\scrH\subset\SS_1^{d-1}$ be an arbitrary open hemisphere, and suppose that $\scrD\subset\SS_1^{d-1}$ has non-empty interior and satisfies $\scrD^\cl\subset\scrH$. Then there is a constant $c_\scrD>0$ such that, for any $\epsilon>0$, there exists an open subset $\scrU_\epsilon\subset\GamG$ and integer $N_\epsilon$ with the property that, for all $\Gamma M\in\scrU_\epsilon$ and $N\geq N_\epsilon$,
\begin{equation}
\scrG_{\scrD,N}(M) \geq c_\scrD \epsilon^{-1}.
\end{equation}
\end{prop}

\begin{proof}
Proposition \ref{prop:unbddNEW} (i) shows that $\scrG_{\scrD}(M_\epsilon) \geq c_\scrD \epsilon^{-1}$ for a sufficiently small $c_\scrD>0$. 
Denote the distinct elements of the set $\{ F_\scrD(M_\epsilon,t) \mid t\in (\lambda,1)\}$ by $0<\varphi_1<\cdots<\varphi_{\scrG_\scrD(M_\epsilon)}<1$. Let $\delta=\max_i (\varphi_{i+1}-\varphi_i)$. 
Then by the continuity established in Proposition \ref{prop:unbddNEW} (ii), there exists a neighbourhood $\scrU_\epsilon\subset\GamG$ of $\Gamma M_\epsilon$ and an $\eta_\epsilon>0$ such that for $t\notin 1+\epsilon\ZZ$, we have that
\begin{equation}
|F_\scrD(M,t')-F_\scrD(M_\epsilon,t)|<\delta/2,
\end{equation} whenever $\Gamma M\in\scrU_\epsilon$ and $t'\in (t-\eta_\epsilon,t+\eta_\epsilon)$. For $N\geq \eta_\epsilon^{-1}$ we can find an integer $n$ so that $\frac{n}{N_+} \in (t-\eta_\epsilon,t+\eta_\epsilon)$. This implies $\scrG_{\scrD,N}(M)\geq \scrG_{\scrD}(M_\epsilon) \geq c_\scrD \epsilon^{-1}$ for all $\Gamma M\in\scrU_\epsilon$ and $N\geq \eta_\epsilon^{-1}$. 
\end{proof}

\begin{proof}[Proof of Theorem \ref{thm:one}]
The proof of Theorem \ref{thm:one} now follows from the same argument as the proof of \cite[Theorem 1]{HaynMar2020}. For $\vec\alpha\in P$, we have that the set $\{ \Gamma A_{N_i+\frac12} (\vec\alpha) \mid i \in\NN\}$ is dense in $\GamG$; see Section 2 of \cite{HaynMar2020} for details. The claim on the limit inferior then follows from Proposition \ref{prop:twoB}, since by density we have infinite returns to a given compact subset. The claim on the limit superior follows from Proposition \ref{lowgapprop}, since by density the above set intersects any given open neighbourhood $\scrU_\epsilon$.
\end{proof}

\begin{proof}[Proof of Theorem \ref{thm:two}]
This is analogous to the proof of to the proof of Theorem 2 in \cite{HaynMar2020}. Let
\begin{equation}\label{phi-def}
	\Phi^s =
	\begin{pmatrix}
		\e^{-s} & 0 \\ 0 & \e^{s/d}\bm{1}_d
	\end{pmatrix}
	\in G.
\end{equation}
A special case of Dani's correspondence \cite[Theorem 2.20]{Dani1985} states that, for any $\vec{\alpha}\in\R^d$, the orbit
\begin{equation}\label{orbit}
	\bigg\{ \Gamma \begin{pmatrix} 1 & \vec{\alpha} \\ 0 & \bm{1}_d \end{pmatrix} \Phi^s \;\bigg|\; s\in\RR_{\geq 0}\bigg\}
\end{equation}
is bounded in $\GamG$ if and only if $\vec{\alpha}$ is badly approximable by $\Q^d$.

Note that for $\ZZ^d M_0=\ZZ^d M_0$ and $\vec\alpha=\vec\alpha_0 M_0$
\begin{equation}
A_{N}(\vec\alpha,\scrL) = \begin{pmatrix} 1 & \vec{\alpha}_0 \\ 0 & \bm{1}_d  \end{pmatrix} \begin{pmatrix} N^{-1} & 0 \\ 0 & N^{1/d}\bm{1}_d\end{pmatrix} \begin{pmatrix} 1 & 0 \\ 0 & M_0  \end{pmatrix} .
\end{equation}
By assumption $\vec{\alpha}$ is badly approximable by $\QQ\scrL$, i.e., $\vec{\alpha}_0$ is badly approximable by $\QQ^d$. It follows from Dani's correspondence that the set $\{ \Gamma A_{N_+}(\vec\alpha_0 M_0,\ZZ^d M_0) \mid N\in\NN\}$
is contained in a compact subset of $\GamG$. The claim then follows from Proposition \ref{prop:twoB}. 
\end{proof}

We now return to the special case $\scrD=\SS_1^{d-1}$, and provide the remaining ingredients for the proof of Theorem \ref{prop:BestPoss}.

\begin{prop}\label{prop:fiver}
Let $\scrD=\SS_1^{d-1}$, and fix a matrix $A_{N_+}(\vec\alpha)$ as in \eqref{eqn.A_NDef} with $N_+:=N+\tfrac12$, $N\in\NN$. Then the following hold.
\begin{enumerate}[{\rm (i)}]
\item For given $n=1,\ldots,N$, the function $t\mapsto F_\scrD(A_{N_+}(\vec\alpha),t)$ is constant on the interval $I_n=(N_+^{-1}(n-\frac12),N_+^{-1} n]$.
\item $F_\scrD$ is continuous at $(\Gamma A_{N_+}(\vec\alpha),t)\in\GamG\times (0,1)$ if $t,1-t\notin N_+^{-1}\ZZ$.
\end{enumerate}
\end{prop}

\begin{proof}
Throughout this proof set $\scrD=\SS_1^{d-1}$. Then
\begin{multline}
F_\scrD(A_{N_+}(\vec\alpha),t)
= N_+^{1/d} \min\big\{| k\vec\alpha + \vec\ell |>0 \\
\;\big|\;  -N_+ t< k< N_+(1-t),\; k\in\ZZ,\; \vec\ell\in\ZZ^d M_0 \big\} .
\end{multline}
The set $(-N_+ t, N_+(1-t))\cap\ZZ$ is independent of the choice of $t\in I_n$; this proves (i). In view of Proposition \ref{prop:cont}, claim (ii) holds if
\begin{equation}\label{exset000777}
	(\ZZ^{d+1} A_{N_+}(\vec\alpha)\setminus\{0\})
	\cap \partial((-t,1-t) \times (0,\kappa]\scrD) = \emptyset ,
\end{equation}
for a fixed choice of $\kappa>\sup_{t\in(0,1)} F_\scrD(A_{N_+}(\vec\alpha),t)$. 
The set
\begin{equation}\label{MD11}
\tilde\scrM(M)= \big\{| \vec v |>0 \;\big|\;  (u,\vec v)\in\ZZ^{d+1} M,\;  |u|\leq 1 \big\}
\end{equation}
is discrete for every fixed $M\in G$ (cf.~\eqref{MD10}), so clearly we can choose $\kappa\notin\tilde\scrM(A_{N_+}(\vec\alpha))$. This means that the lattice $\ZZ^{d+1} A_{N_+}(\vec\alpha)$ does not intersect the set $[-1,1]\times \kappa\scrD$. Furthermore, by the assumption $t,1-t\notin N_+^{-1}\ZZ$, we have 
\begin{equation}\label{exset000777111}
	(\ZZ^{d+1} A_{N_+}(\vec\alpha)\setminus\{0\})
	\cap (\{-t,1-t\} \times (0,\kappa]\scrD) = \emptyset ,
\end{equation}
which establishes \eqref{exset000777}, and hence completes the proof of claim (ii).
\end{proof}

\begin{proof}[Proof of Theorem \ref{prop:BestPoss}.]
In the following $\scrD=\SS_1^{d-1}$. We fix $N\in\NN$ until the last step of the proof. By \eqref{NN1102} and Proposition \ref{prop:fiver} (i) we have
\begin{equation}\label{NN11023}
\delta_{n,N}(\scrD) = N_+^{-1/d} F_\scrD\big(A_{N_+}(\vec\alpha_0,\scrL_0),t_n \big) ,
\end{equation}
for any $t_n\in I_n$, and hence
\begin{equation}
g_{N}(\vec\alpha_0,\scrL_0)= |\{F_\scrD(A_{N_+}(\vec\alpha_0,\scrL_0),t_n)\mid n=1,\ldots,N\}|  .
\end{equation}
Choose $\delta>0$ sufficiently small so that the elements of the set $\{F_\scrD(A_{N_+}(\vec\alpha_0,\scrL_0),t_n)\mid n=1,\ldots,N\}$ are separated by at least $\delta$. Fix any $t_n\in I_n$ such that $t_n,1-t_n \notin N_+^{-1}\ZZ$.
Proposition \ref{prop:fiver} (ii) implies that $F_\scrD$ is continuous at $(\Gamma A_{N_+}(\vec\alpha_0,\scrL_0),t_n)$, for $n=1,\ldots,N$. 
That is, there exists a neighbourhood $\scrU\subset\GamG$ of the point $\Gamma A_{N_+}(\vec\alpha_0,\scrL_0)$ and an $\eta>0$ such that
\begin{equation}
|F_\scrD(M,t')-F_\scrD(A_{N_+}(\vec\alpha_0,\scrL_0),t_n)|<\delta/2,
\end{equation} 
whenever $\Gamma M\in\scrU$ and $t'\in (t_n-\eta,t_n+\eta)$. For every integer $\tilde N\geq \eta^{-1}$ and $n=1,\ldots,N$ we can find a positive integer $m_n\leq \tilde N$ so that $\frac{m_n}{\tilde N_+} \in (t_n-\eta,t_n+\eta)$. This implies that $\scrG_{\scrD,\tilde N}(M)\geq \scrG_{\scrD}(A_{N_+}(\vec\alpha_0,\scrL_0)) =g_{N}(\vec\alpha_0,\scrL_0)$ for all $\Gamma M\in\scrU$ and $\tilde N\geq \eta^{-1}$. 

We can now conclude the proof as for Theorem \ref{thm:one}. Let $\vec\alpha\in P$. The density of the orbit $\{ \Gamma A_{N_i+\frac12} (\vec\alpha,\scrL) \mid i \in\NN\}$ implies
\begin{equation}
\limsup_{i\rar\infty} g_{N_i}(\vec\alpha,\scrL) \geq g_{N}(\vec\alpha_0,\scrL_0) .
\end{equation}
Theorem \ref{prop:BestPoss} now follows by taking the supremum over $N\in\NN$.
\end{proof}

\vspace{.15in}

{\footnotesize
\noindent
AH: Department of Mathematics, University of Houston,\\
Houston, TX, United States.\\
haynes@math.uh.edu\\

\noindent
JM: School of Mathematics, University of Bristol,\\
Bristol, United Kingdom.\\
j.marklof@bristol.ac.uk
}

\end{document}

%% file: 5gaps6.pdf_tex
\begingroup%
  \makeatletter%
  \providecommand\color[2][]{%
    \errmessage{(Inkscape) Color is used for the text in Inkscape, but the package 'color.sty' is not loaded}%
    \renewcommand\color[2][]{}%
  }%
  \providecommand\transparent[1]{%
    \errmessage{(Inkscape) Transparency is used (non-zero) for the text in Inkscape, but the package 'transparent.sty' is not loaded}%
    \renewcommand\transparent[1]{}%
  }%
  \providecommand\rotatebox[2]{#2}%
  \newcommand*\fsize{\dimexpr\f@size pt\relax}%
  \newcommand*\lineheight[1]{\fontsize{\fsize}{#1\fsize}\selectfont}%
  \ifx\svgwidth\undefined%
    \setlength{\unitlength}{251.25bp}%
    \ifx\svgscale\undefined%
      \relax%
    \else%
      \setlength{\unitlength}{\unitlength * \real{\svgscale}}%
    \fi%
  \else%
    \setlength{\unitlength}{\svgwidth}%
  \fi%
  \global\let\svgwidth\undefined%
  \global\let\svgscale\undefined%
  \makeatother%
  \begin{picture}(1,1.00000018)%
    \lineheight{1}%
    \setlength\tabcolsep{0pt}%
    \put(0,0){\includegraphics[width=\unitlength,page=1]{5gaps6.pdf}}%
  \end{picture}%
\endgroup%

%% file: 5gaps5.pdf_tex
\begingroup%
  \makeatletter%
  \providecommand\color[2][]{%
    \errmessage{(Inkscape) Color is used for the text in Inkscape, but the package 'color.sty' is not loaded}%
    \renewcommand\color[2][]{}%
  }%
  \providecommand\transparent[1]{%
    \errmessage{(Inkscape) Transparency is used (non-zero) for the text in Inkscape, but the package 'transparent.sty' is not loaded}%
    \renewcommand\transparent[1]{}%
  }%
  \providecommand\rotatebox[2]{#2}%
  \newcommand*\fsize{\dimexpr\f@size pt\relax}%
  \newcommand*\lineheight[1]{\fontsize{\fsize}{#1\fsize}\selectfont}%
  \ifx\svgwidth\undefined%
    \setlength{\unitlength}{251.24998474bp}%
    \ifx\svgscale\undefined%
      \relax%
    \else%
      \setlength{\unitlength}{\unitlength * \real{\svgscale}}%
    \fi%
  \else%
    \setlength{\unitlength}{\svgwidth}%
  \fi%
  \global\let\svgwidth\undefined%
  \global\let\svgscale\undefined%
  \makeatother%
  \begin{picture}(1,1.00000049)%
    \lineheight{1}%
    \setlength\tabcolsep{0pt}%
    \put(0,0){\includegraphics[width=\unitlength,page=1]{5gaps5.pdf}}%
  \end{picture}%
\endgroup%

%% file: 7gaps3D4.pdf_tex
\begingroup%
  \makeatletter%
  \providecommand\color[2][]{%
    \errmessage{(Inkscape) Color is used for the text in Inkscape, but the package 'color.sty' is not loaded}%
    \renewcommand\color[2][]{}%
  }%
  \providecommand\transparent[1]{%
    \errmessage{(Inkscape) Transparency is used (non-zero) for the text in Inkscape, but the package 'transparent.sty' is not loaded}%
    \renewcommand\transparent[1]{}%
  }%
  \providecommand\rotatebox[2]{#2}%
  \newcommand*\fsize{\dimexpr\f@size pt\relax}%
  \newcommand*\lineheight[1]{\fontsize{\fsize}{#1\fsize}\selectfont}%
  \ifx\svgwidth\undefined%
    \setlength{\unitlength}{313.06519318bp}%
    \ifx\svgscale\undefined%
      \relax%
    \else%
      \setlength{\unitlength}{\unitlength * \real{\svgscale}}%
    \fi%
  \else%
    \setlength{\unitlength}{\svgwidth}%
  \fi%
  \global\let\svgwidth\undefined%
  \global\let\svgscale\undefined%
  \makeatother%
  \begin{picture}(1,0.88042413)%
    \lineheight{1}%
    \setlength\tabcolsep{0pt}%
    \put(0,0){\includegraphics[width=\unitlength,page=1]{7gaps3D4.pdf}}%
  \end{picture}%
\endgroup%

%% file: cone5.pdf_tex
\begingroup%
  \makeatletter%
  \providecommand\color[2][]{%
    \errmessage{(Inkscape) Color is used for the text in Inkscape, but the package 'color.sty' is not loaded}%
    \renewcommand\color[2][]{}%
  }%
  \providecommand\transparent[1]{%
    \errmessage{(Inkscape) Transparency is used (non-zero) for the text in Inkscape, but the package 'transparent.sty' is not loaded}%
    \renewcommand\transparent[1]{}%
  }%
  \providecommand\rotatebox[2]{#2}%
  \newcommand*\fsize{\dimexpr\f@size pt\relax}%
  \newcommand*\lineheight[1]{\fontsize{\fsize}{#1\fsize}\selectfont}%
  \ifx\svgwidth\undefined%
    \setlength{\unitlength}{231.06903076bp}%
    \ifx\svgscale\undefined%
      \relax%
    \else%
      \setlength{\unitlength}{\unitlength * \real{\svgscale}}%
    \fi%
  \else%
    \setlength{\unitlength}{\svgwidth}%
  \fi%
  \global\let\svgwidth\undefined%
  \global\let\svgscale\undefined%
  \makeatother%
  \begin{picture}(1,0.98482848)%
    \lineheight{1}%
    \setlength\tabcolsep{0pt}%
    \put(0,0){\includegraphics[width=\unitlength,page=1]{cone5.pdf}}%
  \end{picture}%
\endgroup%

%% file: cone4.pdf_tex
\begingroup%
  \makeatletter%
  \providecommand\color[2][]{%
    \errmessage{(Inkscape) Color is used for the text in Inkscape, but the package 'color.sty' is not loaded}%
    \renewcommand\color[2][]{}%
  }%
  \providecommand\transparent[1]{%
    \errmessage{(Inkscape) Transparency is used (non-zero) for the text in Inkscape, but the package 'transparent.sty' is not loaded}%
    \renewcommand\transparent[1]{}%
  }%
  \providecommand\rotatebox[2]{#2}%
  \newcommand*\fsize{\dimexpr\f@size pt\relax}%
  \newcommand*\lineheight[1]{\fontsize{\fsize}{#1\fsize}\selectfont}%
  \ifx\svgwidth\undefined%
    \setlength{\unitlength}{206.06833649bp}%
    \ifx\svgscale\undefined%
      \relax%
    \else%
      \setlength{\unitlength}{\unitlength * \real{\svgscale}}%
    \fi%
  \else%
    \setlength{\unitlength}{\svgwidth}%
  \fi%
  \global\let\svgwidth\undefined%
  \global\let\svgscale\undefined%
  \makeatother%
  \begin{picture}(1,0.98849073)%
    \lineheight{1}%
    \setlength\tabcolsep{0pt}%
    \put(0,0){\includegraphics[width=\unitlength,page=1]{cone4.pdf}}%
  \end{picture}%
\endgroup%

%% file: cone6.pdf_tex
\begingroup%
  \makeatletter%
  \providecommand\color[2][]{%
    \errmessage{(Inkscape) Color is used for the text in Inkscape, but the package 'color.sty' is not loaded}%
    \renewcommand\color[2][]{}%
  }%
  \providecommand\transparent[1]{%
    \errmessage{(Inkscape) Transparency is used (non-zero) for the text in Inkscape, but the package 'transparent.sty' is not loaded}%
    \renewcommand\transparent[1]{}%
  }%
  \providecommand\rotatebox[2]{#2}%
  \newcommand*\fsize{\dimexpr\f@size pt\relax}%
  \newcommand*\lineheight[1]{\fontsize{\fsize}{#1\fsize}\selectfont}%
  \ifx\svgwidth\undefined%
    \setlength{\unitlength}{280.92247009bp}%
    \ifx\svgscale\undefined%
      \relax%
    \else%
      \setlength{\unitlength}{\unitlength * \real{\svgscale}}%
    \fi%
  \else%
    \setlength{\unitlength}{\svgwidth}%
  \fi%
  \global\let\svgwidth\undefined%
  \global\let\svgscale\undefined%
  \makeatother%
  \begin{picture}(1,1.40476825)%
    \lineheight{1}%
    \setlength\tabcolsep{0pt}%
    \put(0,0){\includegraphics[width=\unitlength,page=1]{cone6.pdf}}%
  \end{picture}%
\endgroup%

%% file: gapscircles4.pdf_tex
\begingroup%
  \makeatletter%
  \providecommand\color[2][]{%
    \errmessage{(Inkscape) Color is used for the text in Inkscape, but the package 'color.sty' is not loaded}%
    \renewcommand\color[2][]{}%
  }%
  \providecommand\transparent[1]{%
    \errmessage{(Inkscape) Transparency is used (non-zero) for the text in Inkscape, but the package 'transparent.sty' is not loaded}%
    \renewcommand\transparent[1]{}%
  }%
  \providecommand\rotatebox[2]{#2}%
  \newcommand*\fsize{\dimexpr\f@size pt\relax}%
  \newcommand*\lineheight[1]{\fontsize{\fsize}{#1\fsize}\selectfont}%
  \ifx\svgwidth\undefined%
    \setlength{\unitlength}{605.40531921bp}%
    \ifx\svgscale\undefined%
      \relax%
    \else%
      \setlength{\unitlength}{\unitlength * \real{\svgscale}}%
    \fi%
  \else%
    \setlength{\unitlength}{\svgwidth}%
  \fi%
  \global\let\svgwidth\undefined%
  \global\let\svgscale\undefined%
  \makeatother%
  \begin{picture}(1,1.00177622)%
    \lineheight{1}%
    \setlength\tabcolsep{0pt}%
    \put(0,0){\includegraphics[width=\unitlength,page=1]{gapscircles4.pdf}}%
    \put(0.52381919,0.64349162){\color[rgb]{0,0,0}\makebox(0,0)[lt]{\lineheight{1.25}\smash{\begin{tabular}[t]{l}$\vec v_i$\end{tabular}}}}%
    \put(0.32259065,0.59557619){\color[rgb]{0,0,0}\makebox(0,0)[lt]{\lineheight{1.25}\smash{\begin{tabular}[t]{l}$\vec v_j$\end{tabular}}}}%
    \put(0.74166265,0.51456143){\color[rgb]{0,0,0}\makebox(0,0)[lt]{\lineheight{1.25}\smash{\begin{tabular}[t]{l}$\vec v_k$\end{tabular}}}}%
    \put(0.4098477,0.48577391){\color[rgb]{0,0,0}\makebox(0,0)[lt]{\lineheight{1.25}\smash{\begin{tabular}[t]{l}$\color{blue}\vec v_i$\end{tabular}}}}%
    \put(0.20030947,0.41599313){\color[rgb]{0,0,0}\makebox(0,0)[lt]{\lineheight{1.25}\smash{\begin{tabular}[t]{l}$\color{red}-\vec v_j$\end{tabular}}}}%
    \put(0.22886421,0.51633943){\color[rgb]{0,0,0}\makebox(0,0)[lt]{\lineheight{1.25}\smash{\begin{tabular}[t]{l}$\color{darkgreen}-\vec v_k$\end{tabular}}}}%
    \put(0.42808391,0.33699768){\color[rgb]{0,0,0}\makebox(0,0)[lt]{\lineheight{1.25}\smash{\begin{tabular}[t]{l}$\pi/3$\end{tabular}}}}%
    \put(0.03652869,0.45900212){\color[rgb]{0,0,0}\makebox(0,0)[lt]{\lineheight{1.25}\smash{\begin{tabular}[t]{l}$\pi/3$\end{tabular}}}}%
  \end{picture}%
\endgroup%